\documentclass[final]{siamltex}
\usepackage{amsmath,amssymb,enumerate,hyperref,graphicx,bbm,appendix,booktabs,enumitem,verbatim,color,fullpage}
\usepackage{wrapfig,plain}
\usepackage{enumerate}
\usepackage{multirow}
\usepackage{enumitem}
\usepackage{pifont}
\usepackage{algorithm}
\usepackage{algorithmic}
\usepackage{boxedminipage}
\usepackage[normalem]{ulem}
\newcommand{\captionfonts}{\small}
\makeatletter  
\long\def\@makecaption#1#2{%
  \vskip\abovecaptionskip
  \sbox\@tempboxa{{\captionfonts #1: #2}}%
  \ifdim \wd\@tempboxa >\hsize
    {\captionfonts #1: #2\par}
  \else
    \hbox to\hsize{\hfil\box\@tempboxa\hfil}%
  \fi
  \vskip\belowcaptionskip}
\makeatother   
\usepackage{array}
\usepackage{makecell}
\newcolumntype{x}[1]{>{\centering\arraybackslash}p{#1}}
 
\usepackage{amsmath, amssymb, bbm, xspace}

\newcommand{\Real}{\ensuremath{\mathbb{R}}}

\DeclareMathOperator*{\st}{subject\;to}

\def\spose#1{\hbox to 0pt{#1\hss}}

\def\text #1{\hbox{\quad#1\quad}}

\def\nthinsp{\mskip -2   mu}

\def\superstar{^{\raise 0.5pt\hbox{$\nthinsp *$}}}
\def\SUPERSTAR{^{\raise 0.5pt\hbox{$*$}}}

\def\lamstarT {\lambda^{\raise 0.5pt\hbox{$\nthinsp *$}T}}

\def\Uscr{{\cal U}}

\def\Wscr{{\cal W}}

\def\hbar{\skew{4.2}\bar h}

		\def\bkE{{\rm I\kern-.17em E}}
		\def\bkE{\mathbb{E}}
		\def\bk1{{\rm 1\kern-.17em l}}
		\def\bkD{{\rm I\kern-.17em D}}
		\def\bkR{{\rm I\kern-.17em R}}
		\def\bkP{{\rm I\kern-.17em P}}
		\def\bkY{{\bf \kern-.17em Y}}
		\def\bkZ{{\bf \kern-.17em Z}}

		\def\beq{\begin{eqnarray}}
		\def\bc{\begin{center}}
		\def\be{\begin{enumerate}}
		\def\bi{\begin{itemize}}
		\def\bs{\begin{small}}
		\def\bS{\begin{slide}}
		\def\ec{\end{center}}
		\def\ee{\end{enumerate}}
		\def\ei{\end{itemize}}
		\def\es{\end{small}}
		\def\eS{\end{slide}}
		\def\eeq{\end{eqnarray}}

	\def\cp2problem#1#2#3#4{\fbox
		 {\begin{tabular*}{0.9\textwidth}
			{@{}l@{\extracolsep{\fill}}l@{\extracolsep{6pt}}l@{\extracolsep{\fill}}c@{}}
				#1 & & $#4 $ 
			\end{tabular*}}}

\newcommand{\pmat}[1]{\begin{pmatrix} #1 \end{pmatrix}}
		
		\renewcommand{\emph}[1]{\textbf{#1}}

		\def\bk1{{\rm 1\kern-.17em l}}
		\def\bkD{{\rm I\kern-.17em D}}
		\def\bkR{{\rm I\kern-.17em R}}
		\def\bkP{{\rm I\kern-.17em P}}
		
		\def\bkZ{{\bf{Z}}}

\newcommand {\beeq}[1]{\begin{equation}\label{#1}}
\newcommand {\eeeq}{\end{equation}}
\newcommand {\bea}{\begin{eqnarray}}
\newcommand {\eea}{\end{eqnarray}}

\def\texitem#1{\par\smallskip\noindent\hangindent 25pt
               \hbox to 25pt {\hss #1 ~}\ignorespaces}

\def\st{\mbox{subject to}}

\def\Qb{\mathbb{Q}}
\def\Zcr{\mathcal{Z}}
\def\mcr{\mu}
\def\wscr{\mathcal{W}}

 \newcommand{\remove}[1]{}

\def\Real{\mathbb{R}}

\def\argmin{\mathop{\rm argmin}}

\def\argmax{\mathop{\rm argmax}}

\usepackage{verbatim}
\usepackage{ulem}
\author{Yue Xie \and
		Uday V.~Shanbhag\thanks{Xie and Shanbhag are with the
			Department of Industrial and Manufacturing Engineering,
					   Pennsylvania State University, University Park,
					   PA 16802, USA, {\tt\small xieyue1990@gmail.com,
						   udaybag@psu.edu}. Xie and Shanbhag gratefully acknowledge the support of the NSF through the award
NSF CAREER award CMMI-1246887.
 }}

\title{On robust solutions to uncertain linear complementarity problems and their variants
	}
\begin{document}
\maketitle
\begin{abstract}
A popular approach for addressing uncertainty in variational
inequality problems is by solving the expected residual minimization
(ERM) problem~\cite{CFukushima05,ChenWZ12}. This avenue necessitates
distributional information associated with the uncertainty and requires
minimizing nonconvex expectation-valued functions.  We consider a
{distinctly different} approach in the context of uncertain linear
complementarity problems with a view towards obtaining robust solutions.
Specifically, we define a {\em robust} solution to a complementarity
problem as one that minimizes the worst-case of the gap function. In
what we believe is amongst the first efforts to comprehensively address
such problems in a distribution-free environment, we show that {under
specified assumptions on the uncertainty sets}, the robust solutions to
uncertain monotone linear complementarity problem can be tractably
obtained through the solution of a {\em single convex} program. We also
define uncertainty sets that ensure that robust solutions to
non-monotone generalizations can also be obtained by solving convex
programs. More generally, robust counterparts of uncertain non-monotone
LCPs are proven to be low-dimensional nonconvex quadratically
constrained quadratic programs. We show that these problems may be
globally resolved by customizing an existing branching scheme. We
further extend the tractability results to include uncertain affine
variational inequality problems defined over uncertain polyhedral sets
as well as to hierarchical regimes captured by mathematical programs
with uncertain complementarity constraints. Preliminary numerics on
{uncertain linear complementarity and traffic equilibrium problems} 
suggest that the presented avenues hold promise.
\end{abstract}
\maketitle

\section{Introduction}\label{sec:introduction}
The fields of robust control~\cite{DuPa00} and robust optimization
~\cite{bental09robust} have grown immensely
over the last two decades in an effort and are guided by the desire to provide
solutions {\em robust} to parametric uncertainty. To provide a context
for our discussion, we begin by defining a convex optimization problem
 \begin{align}
 \label{UOpt$(u)$} 	\min_{x \in X} \, f(x;u),
\end{align}
where $X\subseteq \Real^n$, $u \in \Uscr \subseteq \Real^L$, $f:X \times \Uscr \to \Real$
is a convex function in $x$ for every $u \in \Uscr$. The resulting
collection of uncertain optimization problems is given by the following
set:
$$ \left\{ \min_{x \in X} \, f(x;u)\right\}_{u \in \Uscr}.$$
Given such a set of problems, one avenue for defining a {\em robust}
solution to this collection of uncertain problems is given by the
solution to the following worst case problem:
\begin{align}
\label{ROpt}\min_{x \in X} \max_{u \in \Uscr} f(x;u).
\end{align}
Robust optimization has grown into an established field and there has
been particular interest in deriving tractable robust counterparts to
\eqref{ROpt}; in particular, can one formulate a {\em single} convex
optimization problem whose solution lies in the set of solutions of
\eqref{ROpt}. Such questions have been investigated in linear, quadratic, and
in more general convex regimes~\cite{bental09robust,bertsimas11theory}
while more recent efforts have considered integer programming
problems~\cite{onn14robust}.

A particularly important class of problems
that includes convex optimization problems is that of variational
inequality problems~\cite{facchinei02finite}. Recall that a variational inequality problem requires an $x \in X$ such that
\begin{align}
(y-x)^T F(x) \geq 0, \qquad \forall y \in X, \label{VI$(X,F)$}
\end{align}
where $F:X \to \Real^n$. {Hereafter, this problem will be denoted by VI$(X,F)$. Moreover when $X$ is a cone, it is
known that VI$(X,F)$ is equivalent to the
complementarity problem CP$(X,F)$ (see ~\cite{facchinei02finite}); The
latter problem} requires an $x$ such that
\begin{align}
X \ni x \perp F(x) \in X^*, \label{CP$(X,F)$}
\end{align}
where $X^*$ denotes the dual cone  defined as $X^*
\triangleq \{y: y^Tx \geq 0, \, x \in X\}$ and $y\perp x$ implies that
$[y]_i[x]_i = 0$ for every $i$.  Such problems have grown increasingly
relevant in control and optimization theory and find application in the modeling
of convex  Nash games in communication
networks~\cite{alpcan05distributed,yin11nash}, traffic equilibrium
problems~\cite{dafermos1980traffic,harker86note}, and spatial 
equilibrium problems~\cite{hobbs86mill,HobbsPang_CSFE}. Naturally, in
almost all of these settings, uncertainty represents a key concern. For
instance, in Nash-Cournot games, the price function of the quantity
being sold may have uncertain parameters while in traffic
equilibrium problems, travel times are rarely known with certainty.
Given such a challenge, one may articulate a {parametrized}
VI$(X,F(\bullet;\tilde u))$ that requires an $x$ such that
\begin{align} (y-x)^T F(x;\tilde u) \geq 0, \qquad \forall y \in X,
\label{unvi}\end{align}
{where $\tilde u$ denotes a belief regarding an uncertain parameter.}
The resulting collection of uncertain variational inequality problems
is given by the following:
\begin{align}
\left\{ \mbox{VI}(X;F(\bullet;u))\right\}_{u \in \Uscr}.
\label{UVI$(X,F,U)$}
\end{align}
In this paper, we consider the setting where $F(x;u) \triangleq M(u)x+
q(u)$ where $M(u) \in \Real^{n \times n}$, $q(u) \in \Real^n$, and $X
\triangleq \Real_n^+$ and the resulting {collection} of affine variational
inequality problem is equivalent to {collection} of uncertain linear complementarity
problems 
{\begin{align}
\left\{ \mbox{LCP}(M(u),q(u))\right\}_{u \in \Uscr},
\label{uLCP$(M,q)$}
\end{align}
where LCP$(M(u),q(u)$ is defined as follows:
\begin{align}
\label{LCP$(u)$} (M(u)x + q(u))^Tx = 0.\\
\notag M(u)x + q(u) \geq 0,\\
\notag x \geq 0.
\end{align}}
{In subsequent sections, we extend these statements to uncertain affine
variational inequality problems (AVI) over polyhedral sets and uncertain
mathematical programs with complementarity constraints (MPCC). More
general forms of $F(x,u)$ will be considered in future research. Given
that LCPs and MPCCs allow for capturing a broad class of application
settings and the paucity of results for capturing uncertain forms of
such problems, we believe that these questions have relevance and
represent a necessary step before proceeding further.} 

Now we briefly touch upon earlier efforts in addressing
this class of problems. In particular, much of the prior work has
focused on the minimization of the expected residual
function (cf.~\cite{CFukushima05,ChenWZ12,chen09robust} and the
		references therein). 
Given a random map $F(x;\xi)$
where $\xi: \Omega \to \Real^d$, $F:X \times \Real^d \to \Real^n$, and
a probability space $(\Omega, {\cal F}, \mathbb{P})$,  based on
the residual function $f(x;\xi)$, the
expected residual minimization problem (ERM) problem is defined as the
following:
\begin{align}
\label{ERM} \min_{x\in X} \mathbb{E}[f(x;\xi)].
\end{align}
Such avenues have derived such solutions for both  monotone
 as well as more general
stochastic variational inequality problems but are
complicated by several challenges:
\begin{enumerate}
\item[(i)]First, such avenues necessitate the availability of a probability
distribution $\mathbb{P}$. 
\item [(ii)] Second, the expected residual minimization problem, given by
ERM, leads to a possibly nonconvex stochastic optimization problem and much of the research has focused on providing
estimators of stationary points to such problems.
\item[(iii)] Third, this approach focuses on minimizing the {\em
	average} or expected residual and may be less capable of providing
	solutions that minimize worst-case residuals unless one employs
	risk-based variants.
\end{enumerate}
In the spirit of robust approaches employed for the resolution of a
range of optimization and control-theoretic problems, we consider an
avenue that requires an uncertainty set. An alternative  of
immense importance not considered here is the scenario-based
approach~\cite{calafiore00,calafiore01} and is left as the subject of
future work. Furthermore, rather
than minimizing the expected residual function, we consider the
miminization of the worst-case residual over this uncertainty set.
Specifically, we make the following contributions:
\begin{enumerate}
\item[(a)] {First, in the context of stochastic linear complementarity
problems with possibly asymmetric positive semidefinite matrices, we show that
the robust counterpart is a {\bf single} convex optimization problem
under varying assumptions on the associated uncertainty set.
}
\item[(b)] Second, we observe that somewhat surprisingly, robust solutions to non-monotone
regimes can also be obtained {\bf tractably} under some conditions, revealing
that such problems are characterized by hidden convexity. More
generally, robust solutions to uncertain non-monotone LCPs are shown to
lead to low-dimensional nonconvex quadratically constrained quadratic
programs. We customize a recently presented branching scheme to allow
for obtaining global solutions to the resulting QCQPs.
\item[(c)] Third, we extend our statements to two sets of
generalizations: (i) Uncertain affine variational inequality problems over uncertain
polyhedral sets; and (ii) Mathematical programs with uncertain
complementarity constraints. 
\item [(d)] These results are supported by preliminary numerics that provide a {comparison across different solvers on a non-monotone uncertain LCP and a traffic equilibrium case study. In the instance of the latter, we
observe qualitative distinctions between ERM solutions and the proposed robust
solutions.}
\end{enumerate}

The remainder of this paper is organized as follows. In Section
\ref{sec:II}, we motivate our study through two applications and provide
an instance of a monotone complementarity problem with arbitrarily high
price of robustness. {In Section~\ref{sec:III}, we first define the
	robust counterpart and subsequenty discuss tractable
reformulations of {both  uncertain monotone LCPs as well as a special case of uncertain
non-monotone LCPs.} Non-monotone generalizations and their global solutions
form the core of Section~\ref{sec:IV} while generalizations to variational
inequality and hierarchical regimes are considered
in Section~\ref{sec:V}. Preliminary numerics are provided in
Section~\ref{sec:VI} and we conclude with a short set of remarks in
Section~\ref{sec:VII}. Finally, to improve readability, we list the key
acronyms and their equation numbers:
\bi
\item[(1)] VI or VI$(X,F(\bullet))$: see \eqref{VI$(X,F)$}.
\item[(2)] CP or CP$(X,F(\bullet))$: see \eqref{CP$(X,F)$}.
\item[(3)] Uncertain LCP or uLCP$(M(u),q(u))$: see \eqref{LCP$(u)$}.
\item[(4)] ERM: see \eqref{ERM}.
\item[(5)] MPCC: see \eqref{MPCC}.
\ei
\section{Motivating examples and applications}\label{sec:II}
In this section, we begin by providing an example of an uncertain linear
complementarity problem {where the lack of robustness has immense
	impact.} and proceed to  discuss two applications that motivate the study
	of uncertain linear complementarity problems.
\subsection{\bf Robust solution vs. non-robust solution}
{We begin by considering the following simple uncertain LCP:}
\begin{align} \label{eLCP$(u)$}
0 \leq \pmat{x \\ y}
\perp  \pmat{M & 0 \\ 0 & S(\xi,\eta)}\pmat{x \\ y} + \pmat{ -q_x \\
	q(u)} \geq 0, \quad \forall (u,\xi,\eta) \in \Uscr_u \times
	\Uscr_{\zeta}
\end{align}
where $M  = \left(I - {1\over (n+1)}ee^T\right) \in \Real^{n \times n}, S(\xi,\eta) = \xi S_1 + \eta S_2, x \in \Real^n, y \in \Real^n, q_x \in \mathbb{R}^n_+, q(u)= ue$, $u \in \Uscr_u \triangleq [0,1]$.
Furthermore, $S_1 = nI_n + e_ne_n^T$ and $S_2 = ee^T + e_ne_n^T$. $e$ denotes the column of ones, $e_n = (1,\hdots, n)^T$ and $\Uscr_{\zeta} \triangleq \{\zeta = (\xi,\eta): \xi + \eta \leq 1, \xi \geq 0, \eta \geq 0 \}$. We begin by noting that a solution to the upper system
$$ 0 \leq x \perp M x - q_x \geq 0 $$
is uniquely defined by $x^*$ where $x^* = (I+ee^T)q_x \geq 0$.
The lower system requires solving the following equation:
$$
0 \leq y \perp S(\xi, \eta) y + q(u) \geq 0.
$$
Since $S(\xi,\eta) \succeq 0$ and $q(u) \geq 0$, it follows that $y
\equiv 0$ is a solution for all $u \in \Uscr_u$ and $(\xi,\eta) \in
\Uscr_{\zeta}$ However, if $\xi = \eta = u =
0$, then any nonnegative $y$ is also a solution implying that there is a
ray of solutions.  Our focus lies in obtaining solutions that minimize
the worst-case residual defined as follows:
 \begin{align*}
& \quad \max_{u \in \Uscr_u,(\xi, \eta)\in \Uscr_\zeta} \pmat{x \\ y}^T \left(\pmat{M & 0 \\ 0 &
		 S(\xi,\eta)}\pmat{x \\ y} +
\pmat{-q_x\\q(u)}\right) 
 = x^T (M x - q_x) + \max_{u \in \Uscr_u,(\xi, \eta)\in \Uscr_\zeta} \left(y^TS(\xi,\eta)y +
	y^Tq(u)\right)\\
& = x^T (M x - q_x) + \max_{u \in \Uscr_u,(\xi, \eta)\in \Uscr_\zeta} \left(\xi y^TS_1y + \eta y^T S_2 y
	+ u y^Te \right)\\
& = x^T (M x - q_x) + \max\{y^T S_1 y, y^T S_2 y\} + y^T e.
\end{align*}
We use this setting to distinguish between a non-robust solution and a
robust solution.\\
 \noindent {\em A non-robust solution:}  {Suppose the realization of $u,\xi,\eta$
 is such that $u = 0,\xi=0,\eta=0$ and the resulting solution is given by
 $(x^*,y_1)$ where $y_1 \geq 0$. Consequently, the worst-case residual
 given by $x^T (M x - q_x) + \max\{y_1^T S_1 y_1, y_1^T S_2 y_1\} + y_1^T e.$ could be arbitrarily high since $y_1$ is any
 nonnegative vector. In effect, a non-robust solution chosen under a single realization can have large
 worst-case residual.} \\
 \noindent {\em A robust solution:} The robust solution of this problem
 is given by $(x^*,0)$ and achieves that worst-case residual equals
 to $x^{*T} (M x^* - q_x)$.
 
\noindent {\bf Remark:} {This example shows that in contrast with a {\em robust solution} of
	an uncertain LCP, {\em non-robust} solutions may lead to
	arbitrarily poor worst-case residual.} 
\subsection{Applications}
\paragraph{\bf Uncertain traffic equilibrium problems}
{A static traffic equilibrium model~\cite{facchinei02finite,harker86note}
captures equilibriating (or steady-state) flows in a traffic network in
which a collection of selfish users attempt to minimize travel costs.
Here, we present a path-based formulation in $\mathcal{N}$. $\mathcal{N}$ denotes the
network while $\mathcal{A}$ represents the associated set of edges. Further, let $\mathcal{O}$ and $\mathcal{D}$ denote the
set of origin and destination nodes, respectively while the set of
origin-destination (OD) pairs is given by $\mathcal{W} \subseteq
\mathcal{O} \times \mathcal{D}$}. Let $\mathcal{P}_w$ denote the set of paths connecting each $w \in
\mathcal{W}$ and $\mathcal{P} = \cup_{w \in \mathcal{W}}\mathcal{P}_w$.
Let $h_p$ denote the flow on path $p \in \mathcal{P}$ while $C_p(h;u)$,
the associated (uncertain) travel cost on $p$, is a function of the
entire vector of flows $h \equiv (h_p)$ and the uncertainty $u \in
\Uscr$. Let
$d_w(v;u)$  represent the uncertain travel demand between O-D pair $w$
and is a function of  $v
\equiv (v_w)$, the vector of minimum travel costs between any
OD pair, and the uncertainty $u \in \Uscr$. Based on Wardrop user equilibrium principle, users
choose a minimum cost path between each O-D pair:
\begin{equation}\label{cp1}
{0} \leq h_p \perp v_w - C_p(h)  \geq {0},\qquad \forall
{w} \in \mathcal{W}, \quad  p \in \mathcal{P}_w.
\end{equation}
{Additionally, the travel costs are related to demand satisfaction
	through this problem}:
\begin{equation}\label{cp2}
0 \leq v_w \perp \sum_{p \in \mathcal{P}_w} h_p - d_w(v) \geq 0,\qquad \forall w \in \mathcal{W},
\end{equation}
Static traffic user equilibrium problem is to solve a pair $(h,u)$
satisfying \eqref{cp1} and \eqref{cp2}, compactly stated as the
following  uncertain complementarity problem:
\begin{align}
0 \leq \pmat{h \\ v} \perp
\pmat{C(h;u) - B^T v \\
Bh - d(v;u)} \geq 0, \quad \forall u \in \Uscr
\end{align}
where $C(h;u)=(C_p(h;u) \mid p \in \mathcal{P})$, $d(v;u)=(d_w(v;u) \mid w \in
		\mathcal{W})$ and $B$ is the (OD pair, path)-incidence matrix $(b_{wp})$:
$$
b_{wp} \triangleq
\left\{ \begin{array}{ll}
1 & \mbox{if } p \in \mathcal{P}_w\\
0 & \mbox{otherwise}.
\end{array}\right.
$$
This represents an uncertain collection of complementarity problems and
we desire an equilibrium $(h,v)$ that is robust to uncertainty.

\paragraph{\bf Uncertain Nash-Cournot games} {Nash-Cournot models for
competitive behavior find application in a variety of settings, including
	the context of networked power markets~\cite{hobbs01linear}}. We
	describe an instance of a single node $N-$player Nash-Cournot game
	in which $N$ players compete for a single good.
	Suppose  player $i$'s uncertain linear cost
function given by $c_i(u)x_i$ where $x_i$ is her production level decision.
Furthermore, the $i$th player's capacity is denoted by $\mbox{cap}_i$.
{We assume that sales of the good are priced using an (uncertain) price
function dependent on aggregate sales $X$ and denoted by $p(X;u)$ where
$u \in \Uscr$}. We restrict our attention to settings where this price
function is affine and defined as follows: $p(X;u) \triangleq a(u) -
b(u)X$ where $a(u),b(u) > 0$,  $X \triangleq
\sum_{i=1}^N x_i$ and $u \in \Uscr$.
	The $i$th agent's problem is given by the following:
\begin{align}
\tag{Player$(x_{-i})$} \min_{x_i \geq 0}  & \ \left(c_i(u) x_i - p(X;u)
		x_i\right) \\
\notag \st &\qquad x_i \leq \mbox{cap}_i.  \qquad (\lambda_i)
\end{align}
The sufficient equilibrium conditions of the Nash-Cournot game are given
by the concatenation of the resulting optimality conditions:
\begin{align*}
0 \leq x_i  & \perp  b(u) (X + x_i)  + \lambda_i + c_i(u) - a(u) \geq 0,
  & \forall  i, \\
0 \leq \lambda_i &  \perp \mbox{cap}_i - x_i \geq 0,  & \forall i.
\end{align*}
The resulting uncertain LCP is given by the following:
\begin{align*}
0 \leq z \perp M(u) + q(u) \geq 0, \quad \forall u \in \Uscr, \mbox{
	where } M(u) \triangleq \pmat{ b(u) (I+ee^T) & I \\
							-I & {\bf 0}}, q(u) \triangleq \pmat{ c(u) -
								a(u) e \\
										\mbox{cap}},\end{align*}
{$e$ denotes the column of ones, $\mbox{cap}$ is the column of
capacities, and $I$ represents the identity matrix.}
\section{Uncertain LCPs with tractable robust counterparts}\label{sec:III}
While optimization
problems admit a natural objective function, such a function is not
immediately available when considering variational inequality
problems. However, one may define a {\em residual/merit function} associated with
VI$(X,F)$. It may be recalled that a function $g(x)$ is a residual/merit function for
VI$(X,F)$ if {the following hold:}
\bi
\item[(i)] {$g(x) \geq 0$} for $x \in X$; 
\item[(ii)] $x \in X$ solves
VI$(X,F)$ if and only if $g(x) = 0$ (see~\cite{larsson93}).
\ei
 When $X$ is
a cone, the problem reduces to a complementarity problem and the
residual/merit function associated with CP$(X,F)$ is given by {the
	following} \textbf{gap function}, defined as follows}: 
\begin{align} \label{Gap function}	
\theta_{\mathrm{gap}} (x,u) \triangleq
\left\{
\begin{array}{ll}
F(x,u)^T x & \mbox{if } F(x,u) \in X^*,\\
+\infty & \mbox{otherwise.}
\end{array}\right.
\end{align}
{Throughout this section, we assume that the set
$X$ is the nonnegative orthant $\Real_+^n$, $F(x,u) \triangleq M(u)x+q(u)$, where $M(u) \in \Real^{n
	\times n}$ and $q(u) \in \Real^n$ for every $u \in \cal U$.
Furthermore, throughout this paper, we utilize the gap function as the
residual function  in developing tractable convex and
relatively low-dimensional robust counterparts of uncertain LCPs with
monotone and non-monotone maps.  Specifically, in Section~\ref{sec:31}, we consider settings where $q(v)$ is an
uncertain vector and $M(u)$ is an uncertain positive semidefinite matrix
with $v \in {\cal V}$ and $u \in \Uscr$. We provide robust counterparts
in regimes where $\Uscr$ and  ${\cal V}$ are either distinct (unrelated)  or
related under varying assumptions on the uncertainty sets. In
Section~\ref{sec:32}, we provide tractable robust counterparts to regimes where
the $M(u)$ is not necessarily positive semidefinite.}

\subsection{{Uncertain monotone LCPs}}  \label{sec:31}
Much of the efforts in the resolution of uncertain variational
inequality problems has considered the minimization of the expected
residual; instead, we pursue a strategy that has defined the field of
robust optimization in that we consider the minimization of the
worst-case gap function (residual) over a prescribed uncertainty set.
While in its original form, such a problem is relatively challenging nonsmooth semi-infinite
optimization problem. Yet, it can be shown that these problems are equivalent to tractable convex programs. By setting $f(x,u) =
\theta_{\mathrm{gap}}(x,u)$,  \eqref{ROpt} can be recast as follows:
\begin{equation}\label{RCP}
\begin{array}{rrl}
\min & \displaystyle \max_{u \in \mathcal{U}} \quad F(x,u)^T x \\
\st & F(x,u) & \in X^* \qquad \forall u \in \mathcal{U},\\
& x & \in X.
\end{array}
\end{equation}
{Before proceeding, it is worth noting that the robust formulation
attempts to find a solution that minimizes the maximal (worst-case)
	value taken by $F(x,u)^Tx$ over the set of solutions that are
	feasible for every $u \in \Uscr.$ In fact, the following
	relationship holds between the optimization problem \eqref{RCP} and
the original uncertain complementarity problem.}
\begin{lemma}
Consider the problem given by \eqref{RCP}. Then $x \in X$
solves
$$ X \ni x \perp F(x,u) \in X^* \mbox{ for all } u \in \Uscr, $$ if and only if
$x$ is a solution of \eqref{RCP} with optimal value zero.
\end{lemma}

{Unfortunately, it is unlikely that there exists an $x$ that solves
	CP$(X,F(\bullet,u))$ for every $u \in U$; instead, we focus on
		deriving tractable counterparts that produce global minimizers
		to \eqref{RCP} which may be rewritten as follows:
\begin{equation}
\label{unlcp}
\begin{array}{rrl}
{\displaystyle \min_{x \geq 0}} &\displaystyle \max_{u \in \mathcal{U}} \ x^T(M(u)x+q(u)) & \\
\st & \displaystyle\min_{u \in \mathcal{U}} \  M_{i \bullet}(u) x  + q_{i}(u)  & \geq
0, \quad \forall i.
\end{array}
\end{equation}
or
\begin{equation}
\label{unlcp2}
\begin{array}{rrl}
{\displaystyle \min_{x \geq 0, t}} & t\\
\st & \displaystyle \max_{u \in \mathcal{U}} \ x^T(M(u)x+q(u)) & \leq t,\\
 & \displaystyle\min_{u \in \mathcal{U}} \ M_{i \bullet}(u) x  + q_{i}(u)  & \geq
0, \quad \forall i.
\end{array}
\end{equation}
{We first consider the development of tractable counterparts of \eqref{unlcp} or
\eqref{unlcp2} when $M(u) \triangleq M$ and $q$ is uncertain.
	Subsequently, we consider the more general setting when both $M$ and
	$q$ are uncertain but are derived from unrelated
	uncertainty sets.  Finally, we assume that both $M$ and $q$ are
	derived from related uncertainty sets.}
\subsubsection{Uncertainty in $q$} 
In this subsection, $q$ is subject to uncertainty and the problem can be
reduced to obtaining a robust solution to an uncertain quadratic program
(cf.~\cite{bental09robust}).  {We define $q(u)$ as follows:}
\begin{align}
\label{qu}
q(u) { \ \triangleq \ } q_{0}+\sum_{l=1}^L u_l q_{l}, \qquad u = (u_1,u_2,\hdots,u_L) \in \mathcal{U},
\end{align}
{where $\Uscr$ is a perturbation set yet to be specified.}
{Consequently, \eqref{unlcp}
may be rewritten as follows}:
{\begin{equation}\label{unlcpq}
\begin{array}{rrl}
\min &  x^T(Mx + q_0)  + \displaystyle\max_{u \in \mathcal{U}}
	 \sum_{l=1}^L u_l(x^T q_l)  & \\
\st  & M_{i \bullet} x + [q_0]_i+ \displaystyle \min_{u \in \mathcal{U}}
\sum_{l=1}^L u_l [q_l]_i  & \geq
0, \, \qquad \forall i\\
& x  & \geq 0.
\end{array}
\end{equation}}
We begin by considering three types of uncertainty sets: $\Uscr_1,
   \Uscr_{\infty}, $ or $\Uscr_2$, where
\begin{align}
\label{qU}
\begin{aligned}
\Uscr_1 & \triangleq \left\{u: \|u\|_1 \leq 1\right\},  \,
\Uscr_2  \triangleq \left\{u: \|u\|_2 \leq 1\right\},
\mbox{ and}\quad \Uscr_{\infty} \triangleq \left\{u: \|u\|_{\infty} \leq 1\right\},
\end{aligned}
\end{align}
The proof of our first tractability result is relatively
straightforward and is inspired by Examples
1.3.2 and 1.3.3 from \cite{bental09robust}. 
\begin{proposition}[{TRC for uLCP$(M,q(u))$}]
Consider the uncertain LCP given by \eqref{unlcpq} where $M$ is a
positive semidefinite matrix.  Let $\mathcal{U}$ be defined as $\Uscr_1,
		 \Uscr_2$ or $\Uscr_\infty$ as specified by \eqref{qU}. Then \eqref{unlcpq} can be reformulated as a tractable program.
\end{proposition}
\begin{proof}
\begin{enumerate}
\item[(i)] $\Uscr := \Uscr_{\infty}$: We begin by noting that $\|w\|_1 =
\displaystyle \max_{\|\eta\|_{\infty} \leq 1} w^T\eta$. Consequently, the
term $\sum_{l=1}^L u_l x^T q_l$ has maximal value $\sum_{l=1}^L |x^T
q_l|$. {Furthermore, we have that  
	$$ \displaystyle \min_{u \in \Uscr_{\infty}}
\sum_{l=1}^L u_l [q_l]_i = \min_{u \in \Uscr_{\infty}}
		u^Tq = -\max_{u\in \Uscr_\infty} (-u^Tq) = -\left\|\pmat{q_1 \\
			\vdots \\ q_L}\right\|_1 = -\sum_{l=1}^L |q_l|. $$} 
{Consequently, \eqref{unlcpq} is equivalent to the following:}
\begin{align}\label{unlcpq2}
\begin{array}{rrl}
\min & x^T(Mx+q_{0})+\sum_{l=1}^L |q_{l}^T x|\\
\st & Mx+q_{0}-\sum_{l=1}^L |q_{l}| & \geq 0,\\
& x & \geq 0.
\end{array}
\end{align}
Finally, by adding additional variables, \eqref{unlcpq2} can be rewritten
as a convex quadratic program (QP):
$$
\begin{array}{rrl}
\min & x^T (M x + q_0) + \sum_{l=1}^L t_l \\
\st & t_l \geq q_{l}^T x & \geq -t_l, \quad \forall l,\\
& Mx+q_{0}-\sum_{l=1}^L |q_{l}| &\geq 0,\\
& x &\geq 0.
\end{array}
$$
\item[(ii)] $\Uscr := \Uscr_{1}$: We proceed in a fashion similar to
(i) and begin by recalling that 
${\displaystyle \max_{\|\eta\|_1 \leq 1}} \eta^T w = \|w\|_\infty$,
	leading to the folliwing simplification:
$$\max_{u \in \Uscr_1} \sum_{l=1}^L
u_l x^T q_l = \left\| \pmat{x^Tq_1 \\ \vdots \\ x^Tq_L}\right\|_{\infty}
=\displaystyle \max_{l \in \{1,\hdots,L\}} | x^T q_l
|.$$ Similarly, ${\displaystyle \min_{\|\eta\|_1 \leq 1}}
\eta^T w = -\|w \|_\infty$ and $\sum_{l=1}^L u_l (q_l)_i$ has minimal value $-\max_{l
	\in \{1,\hdots,L\}} | (q_l)_i |$. {Consequently,  the constraints
	reduce to \begin{align*}
	 & \qquad M_{i\bullet} x+(q_{0})_i - \displaystyle \max_{l\in \{1,\hdots,L\}} |
	[q_l]_i | \geq 0,&   \forall i \\
 \Leftrightarrow & \qquad M_{i\bullet} x+(q_{0})_i -  |
(q_l)_i |   \geq 0, \qquad & \forall i, l  \\
	 \Leftrightarrow & \qquad Mx+q_{0}- |q_{l}|   \geq 0,& \forall l.
	\end{align*}} Similarly, the objective function can be
	stated as $x^T (M x + q_0) + \max_{l\in \{1,\hdots,L\}} | x^T q_l
	|$. By adding a variable $t$, this problem may be reformulated as
	the following convex quadratic program:
	$$
\begin{array}{rrll}
\min & x^T(Mx+q_{0})+t &&\\
\st & t \geq q_l ^T x & \geq -t, \quad &\forall l\\
& Mx+q_{0}- |q_{l}| & \geq 0, \quad & \forall l\\
& x & \geq 0.&
\end{array}
$$
\item[(iii)] $\Uscr:=\Uscr_2$: By leveraging Example 1.3.3 in
\cite{bental09robust}, it is seen that
$$
\max_{\| \eta \|_2 \leq 1} \eta^T w = \frac{w^Tw}{\|w\|_2} = \|w\|_2.
$$
As a result, $\sum_{l=1}^L u_l x^T q_l$ has maximal value $
\sqrt{{\sum_{l=1}^L (q_{l}^T x)^2}}$ while
$$
\min_{\| \eta \|_2 \leq r} \eta^T w = - \frac{w^Tw}{\|w\|_2} =
- \|w\|_2,
$$
indicating that  $\sum_{l=1}^L u_l (q_l)_i$ has minimal value $-
\sqrt{{\sum_{l=1}^L (q_{l})_{i}^2}}$.
Then \eqref{unlcpq} may be rewritten as:
$$
\begin{array}{rrl}
\mbox{min} & x^T(Mx+q_{0}) &+\sqrt{\sum_{l=1}^L (q_{l}^T x)^2}\\
\mbox{subject to} & Mx+q_{0}- q & \geq 0,\\
& x & \geq 0.
\end{array}
$$
where $q_{i} = \sqrt{{\sum_{l=1}^L (q_{l})_{i}^2}}$. By adding an extra
variable, we obtain convex program with a quadratic objective and a
conic quadratic inequality.$$
\begin{array}{rrl}
\mbox{min} & x^T (M x + q_0) + t \\
\mbox{subject to} &  \sqrt{{\sum_{l=1}^L (q_l^T x)^2}} & \leq t,\\
& M x + q_0 - q & \geq 0,\\
& x & \geq 0.
\end{array}
$$
\end{enumerate}
\end{proof}
Next, we present a more result where the uncertainty set is
captured by a more general convex set. Specifically, $\Uscr:= \Uscr_c$
where
\begin{equation} \label{defUc}
\hspace{-0.1in} \mathcal{U}_c \triangleq \left\{u \in \Real^L : \exists \nu \in \Real^k
: P u + Q \nu + p \in K \subseteq \Real ^ n\right\},
\end{equation}
$K$ is a cone, $P$ and $Q$ are given matrices, and $p$ is a
	given vector. 
\begin{proposition} [RC of uLCP$(M,q(u))$ when $\Uscr:=\Uscr_c$]
Consider the uncertain LCP given by \eqref{unlcpq} where $M$ is a
positive semidefinite matrix. Let $\mathcal{U} :=\Uscr_c$, where
$\Uscr_c$ is defined as \eqref{defUc}. Suppose one of the following
holds:
\begin{enumerate}
\item[(i)] $K$ is a polyhedral cone;
\item[(ii)] $K$ is a convex cone and  the following holds:
\begin{equation}\label{slat_pt}
\exists(\bar{u},\bar{\nu}) \mbox{ such that } P\bar{u} + Q\bar{\nu} + p \in \mbox{int} (K).
\end{equation}
\end{enumerate}
Then the robust counterpart of \eqref{unlcpq} is a finite dimensional
convex conic program {given by the following: 
$$
\begin{array}{rrlr}
\min & x^TM_0x + q_0^Tx+ p^Ty \\
\textrm{\em subject to} & (P^T y)_l + x^T q_l & = 0, & \forall l = 1,\hdots,L \\
& Q^T y & = 0,\\
& y & \in K^*,\\
& M_{i \bullet}x + [q_0]_i - p^T z_i & \geq 0,& \forall i = 1,\hdots, n \\
& [P_{\bullet l}]^T z_i & = [q_l]_i, &\forall i = 1,\hdots, n, l = 1,\hdots,L\\
& Q^T z_i & =0, &\forall i = 1,\hdots, n\\
& z_i & \in K^*,&\forall i = 1,\hdots, n\\
& x & \geq 0.
\end{array}
$$}
\end{proposition}
The proof follows from Theorem 1.3.4
	and Proposition 6.2.1 from \cite{bental09robust} and is omitted.

{\bf Remark:} If $K$ is {chosen to be the nonnegative orthant}, the uncertainty set is a polyhedron given
$Q=0$. {Both ${\cal U}_1$ and ${\cal U}_{\infty}$  are included in
this general case}. If $K$ is chosen to be the second-order cone, a special case of
the perturbation set $\Uscr_c$ is a ball. Under both circumstances, the problem is
tractable.  Notice that nonnegative orthants and Lorentz cones are
self-dual. {When $K$ is chosen to be $\mathbb{R}^n_+$,
	\eqref{unlcpq} reduces to a convex
quadratic program (QP). If $K$ is chosen to be $L^n$,
\eqref{unlcpq} can be recast as a convex quadratically
constrained quadratic program (QCQP)}.

\subsubsection{Uncertainty in both $M$ and $q$ {under an independence
assumption}}{Next, we consider the setting where both $M$  and $q$
are uncertain but the sources of uncertainty are independent.  This is a
	somewhat more challenging problem and  a direct application of the
	results from robust quadratic programming appears difficult.\\}
	{Recall that the map $F$ is said to be monotone over a set $X$ if
		the following holds:
$$
(F(x)-F(y))^T(x-y) \geq 0, \qquad \forall x,y \in X.
$$
{Additionally,  $Mx + q$ is monotone over $\Real_+^n$ if and only if $\frac{1}{2}(M + M^T)$ is positive semidefinite (cf.~\cite{facchinei02finite}). Without loss of generality, we assume that $M$ or $M(u)$ is sysmetric through this subsection, if not, we may always replace
		the matrices by their symmetrized counterparts.} 
For the present, we assume that $q$ is deterministic and reformulate \eqref{unlcp} as follows}:
\begin{equation}\label{unlcpM}
\begin{array}{rrl}
\mbox{min} & t\\
\mbox{subject to} & x^T (M(u) x + q) & \leq \  t, \qquad \forall u \in \mathcal{U},\\
& M(u) x + q & \geq \ 0, \qquad \forall u \in \mathcal{U},\\
& x & \geq \ 0,
\end{array}
\end{equation}
{where $M(u)$ is defined as follows:
\begin{align}\label{defM}
\begin{aligned}
M(u) & \triangleq M_0 + \sum_{l=1}^L u_l M_l, \\
	M_0 & \succeq 0,\\
M_l  & \succeq 0, \qquad  l=1,\hdots,L.
\end{aligned}
\end{align}}
Of course, $M_l, l = 0,\hdots,L$ have also been assumed to be symmetric.
We now present a tractability result for nonnegative restrictions of
	$\Uscr^r_1$ and $\Uscr^r_{\infty}$ defined as follows:
\begin{align}
\Uscr^r_\infty \triangleq \{u : \| u \|_\infty \leq 1, u \geq 0\} \mbox{
	and } \Uscr^r_1 \triangleq \{u : \| u \|_1 \leq 1, u \geq
	0\}.\label{nonneg_U}
\end{align}
{Note that under the definitions of $\Uscr^r_{\infty}$ and $\Uscr^r_{\infty}$, $M(u)$ is always positive semidefinite. This implies that \ref{unlcpM} is convex for each $u \in \Uscr$.}
\begin{proposition} [TRC for uLCP$(M(u),q)$ for $\Uscr_1^r,
	\Uscr^r_\infty$] Consider the
problem \eqref{unlcpM} where $M(u)$ is defined by \eqref{defM} and
$\mathcal{U}$ is chosen either to be $\Uscr^r_{\infty}$ or $\Uscr^r_1$, both
of which are defined in \eqref{nonneg_U}.  Then the uncertain LCP has a tractable robust
counterpart, given by a QP and a QCQP, respectively.
\end{proposition}
\begin{proof}
\begin{itemize}
\item[(a)] $\Uscr:=\Uscr_{\infty}^r$: 
We first 
derive the robust counterpart of the following constraint:
$$x^T M_0 x + \sum_{l=1}^L u_l x^T M_l x + x^T q \leq t,\qquad \forall u \in
\mathcal{U}_{\infty}^r. $$
But this can be equivalently stated as
$$x^T M_0 x + \max_{u \in \Uscr^r_{\infty}} \left[ \sum_{l=1}^L u_l x^T
M_l x \right]+ x^T q \leq
t.$$
We now evaluate the maximum in the right hand side:
\begin{align*}
\max_{u \in \Uscr^r_{\infty}} \left[ \sum_{l=1}^L u_l x^T
M_l x \right ] = \sum_{l=1}^L \max_{u_l \in [0,1]} \left[ u_l x^T
M_l x \right ] = \sum_{l=1}^L \max(x^TM_l x,0) =  \sum_{l=1}^L x^TM_l x, 
\end{align*}
where the last equality is a consequence of applying the positive
semidefiniteness of $M_l$ for $l = 1, \hdots, L$.  
Consequently, the robust counterpart of \eqref{unlcpM} can be stated as
follows: 
$$
\begin{array}{rrl}
\mbox{min} & t\\
\mbox{subject to} &  x^T (M_0 + \sum_{l=1}^L  M_l) x + x^T q & \leq t,\\
& \displaystyle \min_{u \in \Uscr^r_{\infty}}\left[\left( \sum_{l=1}^L u_l
		[M_l]_{i\bullet}x\right) \right]  + [M_0]_{i\bullet}x + q_i  &
\geq 0, \quad \forall i\\
& x & \geq 0.
\end{array}
$$
We may now simplify the second constraint as follows:
$$\min_{u \in \Uscr^r_{\infty}}\left[ \sum_{l=1}^L u_l
		[M_l]_{i\bullet}x \right] = \sum_{l=1}^L \min_{u_l \in [0,1]}
u_l [M_l]_{i\bullet}x = \sum_{i=1}^L v_l, $$
where 
$$ v_l = \begin{cases}
		[M_l]_{i\bullet}x,  & 	\mbox{ if } [M_l]_{i\bullet}x < 0 \\
		0, 					&   \mbox{ if } [M_l]_{i\bullet}x \geq 0.
		\end{cases} $$
As a consequence, $v_l = -\max(-[M_l]_{i\bullet}x,0)$ for $l = 1, \hdots, L$ and
the TRC may be rewritten as follows:
$$
\begin{array}{rrl}
\mbox{min} & t\\
\mbox{subject to} &  x^T (M_0 + \sum_{l=1}^L  M_l) x + x^T q & \leq t,\\
& \displaystyle  -\sum_{l=1}^L\max(-[M_l]_{i\bullet}x,0)
  + [M_0]_{i\bullet}x + q_i  & \geq 0, \qquad \forall i\\
& x & \geq 0.
\end{array}
$$
{Through the addition of  variables $z_1, \hdots, z_L$, we may eliminate
the max. function, leading to the following quadratic program (QP):} 
$$
\begin{array}{rrl}
\mbox{min} & x^T (M_0 + \sum_{l=1}^L  M_l) x + x^T q\\
\mbox{subject to} & M_0 x + q    -  \sum_{l=1}^L z_l & \geq 0,\\
& M_l x + z_l & \geq 0, \quad \forall l,\\
& z_l & \geq 0,  \quad\forall l, \\
& x & \geq 0.
\end{array}
$$
\item[(b)] $\Uscr:=\Uscr_1^r$: 
In an analogous fashion, when $\Uscr := \Uscr_1^r$, we have the
	following sequence of  equivalence statements for the quadratic constraint:
\begin{align*}
& \qquad x^TM_0x + \max_{u \in \Uscr_1^r} \left[ \sum_{l=1}^L u_lx^TM_l
x\right] + x^Tq   \leq t &\\
& \Leftrightarrow \quad x^TM_0x + \max_{l \in \{1, \hdots, L\}} 
\left[\max(x^TM_lx,0)\right] + x^Tq   \leq t &\\
& \Leftrightarrow \quad x^TM_0x +  \max_{l \in \{1, \hdots, L\}}
[x^TM_lx] + x^Tq   \leq t, &   \\
& \Leftrightarrow \quad x^TM_0x +  
x^TM_lx + x^Tq   \leq t, & \quad l = 1, \hdots, L, 
\end{align*}
where the second equivalence statement follows from the positive
	semidefiniteness of $M_l$. The semi-infinite linear constraint can
	be reformulated as follows:
\begin{align*}
& \qquad \displaystyle \min_{u \in \Uscr_1^r}\left[ \sum_{l=1}^L u_l
		[M_l]_{i\bullet}x \right]  + [M_0]_{i\bullet}x + q_i  
	\geq 0, \, & \forall i\\
& \Leftrightarrow \qquad\displaystyle  -\max_{l \in \{1, \hdots, L\}}
\max(-[M_l]_{i\bullet}x,0)  + [M_0]_{i\bullet}x + q_i  
	\geq 0, \, & \forall i\\
& \Leftrightarrow \qquad\displaystyle  
\max(-[M_l]_{i\bullet}x,0)  \leq [M_0]_{i\bullet}x + q_i  \, & \forall i,l\\
& \Leftrightarrow \qquad\displaystyle \max(-M_lx,0) \leq [M_0]x + q, & l = 1, \hdots, L.
\end{align*}
Finally, by the addition of a variable $z$, we obtain the following
QCQP:
$$
\begin{array}{rrl}
\mbox{min} & t\\
\mbox{subject to} &  x^T (M_0 + M_l) x + x^T q & \leq t,  \quad \forall l\\
& M_0 x + q { \ - \ } z & \geq 0,\\
& {M_l x} + z & \geq 0, \quad \forall l\\
& z & \geq 0,\\
& x & \geq 0.
\end{array}
$$
\end{itemize}
\end{proof}

\noindent{\bf Remark:} Note that while we do not explicitly consider the case when $q$ is
	also uncertain, this may be easily introduced when the uncertainty
	set that prescribes $M(u)$ is  unrelated to that producing $q(u)$.
	On this occasion, we may address each term individually, as in the
	prior subsection. Next, we consider the possibility that $M$ and $q$
	are derived from the same uncertainty sets.

\subsubsection{Uncertainty in {$M$ and $q$ under a dependence
	assumption}}
Next, we extend the realm of applicability of the
tractability result to accommodate uncertainty sets that are more
general than \eqref{defM}. Specifically, we employ an uncertainty set that relies on computing the Cholesky Factorization of $M$, defined next as adopted in~\cite{bertsimas11theory}:
\begin{align}\label{defMC}
\hspace{-0.1in}\begin{aligned}
& \mathcal{U_A}  \triangleq \left\{ (M,q) \mid M = A^T A, A = A_0 + \sum_{l=1}^L \xi_l A_l,q = q_0 + \sum_{l=1}^L \xi_l q_l , \| \xi \|_2 \leq 1\right\}.
\end{aligned}
\end{align}
{We explore this construction {since} it allows for developing tractable robust counterparts when $M(u)$ and $q(u)$ are related.} 
Consequently, \eqref{unlcp} may be recast as follows:
\begin{equation}\label{unlcpCF}
\begin{array}{rrl}
\mbox{min} & t\\
\mbox{subject to} & x^T (Mx + q) & \leq t, \qquad \forall (M,q) \in \mathcal{U_A},\\
& Mx + q & \geq 0, \qquad \forall (M,q) \in \mathcal{U_A},\\
& x & \geq 0.
\end{array}
\end{equation}
For the sake of convenience, we write the first constraint as
\begin{align}\label{cons1} x^TMx + 2x^T\left(\frac{q}{2} \right) - t \leq 0, \qquad \forall (M,q) \in
\Uscr_A \end{align}
{The tractability of constraint \eqref{cons1} follows from 
Theorem 2.3 in \cite{bertsimas11theory} and is formalized without
a proof.}
\begin{lemma}\label{lemmacons1}
	Consider the constraint \eqref{cons1} where $\Uscr_A$ is defined by
\eqref{defMC}. Then the tractable
	counterpart of this constraint is given by \eqref{cholcon}:
	\begin{align} \label{cholcon}
	\left(
	\begin{array}{ccccc}
	-q_0^T x + t - \tau & -\frac{1}{2}q_1^T x, & \hdots, & -\frac{1}{2}q_L^T x & (A_0 x)^T \\
	-\frac{1}{2}q_1^T x & \tau & & & (A_1x)^T\\
	\vdots & & \ddots & &\vdots\\
	-\frac{1}{2}q_L^T x & & & \tau & (A_L x)^T\\
	A_0 x & A_1 x ,& \hdots, & A_L x & I \\
	\end{array}
	\right)
	\succeq 0.
	\end{align}
\end{lemma}
{However, it is more challenging to construct a robust
counterpart of the constraint given by \eqref{cons2}}:
\begin{align}\label{cons2}
Mx + q \geq 0, \qquad \forall (M, q) \in \mathcal{U}_A.
\end{align}
In fact, this is the key departure from the result provided in ~\cite{bertsimas11theory}.
For purposes  of convenience and clarity, we rewrite $\mathcal{U}_A$ in terms of $A_0, q_0$ and $A_l ,q_l,l=1,\hdots,L$:
\begin{align} \label{defMC1}
\mathcal{U}_A & \triangleq \left\{(M,q)
\mid M = A_0^T A_0 + \sum_{l=1}^L (A_l^T A_0 + A_0^T A_l) \xi_l  +  \sum_{l < m}(A_l^T A_m + A_m^T A_l) \xi_l \xi_m +
		\sum_{l=1}^L A_l^T A_l \xi_l^2, \right.\notag\\
		& \left. q = q_0 + \sum_{l=1}^L q_l
		\xi_l, \qquad \| \xi \|_2 \leq 1 \right\}.
\end{align}
{We may utilize \cite[Lemma 14.3.7]{bental09robust} in deriving the
tractability of \eqref{cons2}.}
{
\begin{lemma}\label{propcons2}
Consider the constraint \eqref{cons2} where $\Uscr_A$ is defined by
\eqref{defMC1}.  Then the semi-infinite constraint has a tractable robust  counterpart, which will be presented as \eqref{SDPTRC}.
\end{lemma}
}
\begin{proof}
We begin by noticing that obtaining a feasible solution of \eqref{cons2}
requires solving the following $i$th optimization problem for $i = 1, \hdots, n$:
\begin{align} \notag
\min \quad & \sum_{l=1}^L \left( [A_l^T A_0 + A_0^T A_l]_{i\bullet} x +
		[q_l]_i\right) \xi_l  + \sum_{1 \leq l < m \leq L}  \left[A_l^T A_m + A_m^T
A_l\right]_{i\bullet} x \xi_l \xi_m +\sum_{l=1}^L \left[A_l^T A_l\right]_{i\bullet} x \xi_l^2, \\
\st  \quad & \| \xi \|_2 \leq 1. \label{cholprob}
\end{align}
{We may compactly rewrite \eqref{cholprob} as follows:}
\begin{align}\label{nonaffun1}
\begin{aligned}
\min \quad & b_i(x)^T \xi + \xi^T C_i(x) \xi\\
\st \quad & \| \xi \|_2 \leq 1,
\end{aligned}
\end{align}
where $b_i: \mathbb{R}^n \rightarrow \mathbb{R}^L, C_i: \mathbb{R}^n
	\rightarrow \mathbb{R}^{L \times L}$ are all linear functions of
	$x$. We now define the following:
\begin{align}
\begin{aligned}
& \widehat{\xi} = \pmat{ & \xi^T \\ \xi & \xi \xi^T}, M_i(x)=\pmat{&
	\frac{1}{2} b_i^T(x) \\ \frac{1}{2} b_i(x) & C_i(x)}, \mbox{ and }  \mathcal{Z} = \left\{ \pmat{ & \xi^T \\ \xi & \xi \xi^T} \mid \| \xi \|_2 \leq 1 \right\}
\end{aligned}
\end{align}
{Then the QCQP \eqref{nonaffun1} is equivalent to the
	following matrix optimization problem}:
\begin{align}\label{nonaffun2}
\min_{\widehat{\xi} \in \mathcal{Z}} \quad & \langle\widehat{\xi}, M_i(x)\rangle
\end{align}
where $\langle A, B\rangle = Tr(A^TB)$. {Since the objective function is
	linear in $\widehat \xi$}, we may extend the feasible region
	$\mathcal{Z}$ to its convex hull $\widehat{\mathcal{Z}}$ which is
	given by $ \widehat {\cal Z} \triangleq 
	\mbox{conv}\{\mathcal{Z}\}$. By Lemma 14.3.7 from \cite{bental09robust}, we have that
$$
\widehat{\mathcal{Z}} = \left\{ \widehat{\xi} = \pmat{& w^T \\ w & W} \in \textbf{S}^{L + 1} \mid \pmat{1 & w^T \\ w & W} \succeq 0, Tr(W) \leq 1\right\},
$$
{where ${\bf S}^{N+1}$ represents the cone of symmetric positive semidefinite
matrices. Using variable replacement, \eqref{nonaffun2} is equivalent to the
following semidefinite program}:
\begin{align}\label{nonaffun3}
\begin{aligned}
\min_{
 X \in \textbf{S}^{L+1}} \quad  \langle X, M_i(x)\rangle & \\
\st \quad  \left\langle X, \pmat{0 & \\  & I} \right\rangle & \leq 1,
 \quad
 \left\langle X, \pmat{1 & \\ &} \right\rangle  = 1.
\end{aligned}
\end{align}
The feasible region of \eqref{nonaffun3} obviously acquires a
	nonempty relative interior. Therefore its dual optimum can be
	obtained and thus we can reformulate the $i$th constraint in
	\eqref{cons2} as the following SDP constraint:
\begin{align} \label{SDPTRC}
\begin{aligned}
y_{i,1} + y_{i,2} + a_i (x) & \geq 0,\\
y_{i,1}\pmat{0 & \\  & I} + y_{i,2}\pmat{1 & \\ &} & \preceq M_i(x),\\
y_{i,1} & \leq 0.
\end{aligned}
\end{align}
where $ a_i (x) = [A_0^TA_0]_{i\bullet}x + [q_0]_i$.
\end{proof}

Our main result of this subsection can be seen to follow from
	Lemma~\ref{lemmacons1} and Lemma~\ref{propcons2}. 
\begin{theorem}[TRC for uLCP$(M,q)$ for $(M,q) \in
	\Uscr_A$]\label{RC_unlcp_mon}
Consider the uncertain LCP \eqref{unlcpCF}. Then this semi-infinite
program has a tractable robust counterpart {given by the following
	SDP: 
$$
\begin{array}{rrlr}
\displaystyle \min_{t,\tau,x,y_1,y_2} & t \\
\textrm{\em subject to}  & \tau \pmat{-1&&\\&I_L&\\&&0_{n\times n}} + t \pmat{1&\\&0_{(L+n)\times(L+n)}}+M_0(x) &\succeq 0,\\
& y_{i,1}\pmat{0 & \\  & -I_L} + y_{i,2}\pmat{-1 & \\ & 0_{L\times L} }
+ M_i(x) & \succeq 0, & \forall i = 1, \hdots, n \\
& y_{i,1} + y_{i,2} + a_i (x) & \geq 0,  &\forall i = 1, \hdots, n\\
& y_{i,1} & \leq 0, &\forall i = 1, \hdots, n\\
& x & \geq 0.
\end{array}
$$}
where $M_0: \mathbb{R}^n \rightarrow \mathbb{R}^{(L+n+1)\times(L+n+1)}, M_i: \mathbb{R}^n \rightarrow \mathbb{R}^{(L+1)\times(L+1)}, a_i: \mathbb{R}^n \rightarrow \mathbb{R}$ are all linear functions of $x$.
\end{theorem}

\noindent {\bf Remark:}
\bi
\item[(i)] When the uncertainty set $\|\xi\|_2 \leq 1$ is replaced by
	either $\|\xi\|_1 \leq 1$ or $\|\xi\|_{\infty} \leq 1$,
	Lemma~\ref{lemmacons1} does not hold. What we
{may} do instead is to enlarge the uncertainty set to get a tractable robust
counterpart. In {the} case of $\|\xi\|_{\infty} \leq 1$, \cite[Lemma
14.3.9]{bental09robust}  provides a semidefinite representable set that
contains $\widehat{\mathcal{Z}}$. On the other hand, we
{may enlarge $\|\xi\|_1 \leq 1$ or $\|\xi\|_{\infty} \leq 1$ to their circumscribed
	spheres representing a scaling of $\Uscr_2$ allowing for the
	construction of tractable robust counterparts} of \eqref{cons2}.
\item[(ii)] {We note that \cite{wu2011robust} claims a similar result
(Theorem 3.2) as Theorem ~\ref{RC_unlcp_mon}. However, there appears to
be an issue in that the tractability of ~\eqref{cons2}
is not proved and does not seem to follow directly.}\ei

\subsection{Tractable uncertain non-monotone LCPs} \label{sec:32}
When the matrix $\frac{1}{2}(M(u)^T + M(u))$ is not {necessarily}
positive semidefinite for {every} $u \in \mathcal{U}$, $M(u)x+q(u)$
	is no longer monotone for every $u$. {Consequently,} the problem
	\eqref{unlcpM} no longer has convex constraints for {every}
realization of $u$. {As we proceed to show,  we may still obtain a tractable robust counterpart
under a suitably defined uncertainty set on $M(u)$ with the
caveat that $M(u)$ is unrelated to $q(u)$.} We begin by defining the
uncertainty set for $M(u)$.
\begin{align}\label{defM2}
\begin{aligned}
M(u) & \triangleq M_0 + \sum_{l=1}^L u_l M_l,\\
\frac{1}{2}(M_0 + M_0^T) & \succeq 0,\\
\frac{1}{2}(M_l + M_l^T) & \succeq 0, \qquad \forall l=1,\hdots,L.
\end{aligned}
\end{align}
{Without loss of generality, we assume that $M(u), M_0$ and $M_l$ are
	symmetric for $l = 1, \hdots, L$; if not, we may always replace
		the matrices by their symmetrized counterparts.  {Thus far, the definition of $M(u)$ is the same as \eqref{defM}. However, we will take $\Uscr$ as $\Uscr_1,\Uscr_2,\Uscr_\infty$. This approach leads to indefinite matrix thus nonmonotone affine map for some scenarios.} The tractability of the robust
		counterpart of the uncertain nonmonotone LCP is proved next.}
\begin{proposition} [{TRC for non-monotone
	uLCP$(M(u),q)$}]\label{prop_nmon_lcp} Consider the
problem \eqref{unlcpM}. {Suppose $M$ is defined by \eqref{defM2} and
$\mathcal{U}$ is either $\Uscr_{\infty}, \Uscr_1$ or $\Uscr_2$.  
Then this problem admits a tractable robust
counterpart.}
\end{proposition}
\begin{proof} \begin{itemize}
\item[(a)] $\mathcal{U}: = \{u : \| u \|_\infty \leq 1\}$: We begin
	by determining the robust counterpart of the following constraint:
\begin{align}
\label{U1defM}
x^T M_0 x + \sum_{l=1}^L u_l x^T M_l x + x^T q \leq t,\quad \forall u \in \mathcal{U}. \end{align}
This may be equivalently stated as
$$x^T M_0 x + \max_{u \in \Uscr_{\infty}} \left[\sum_{l=1}^L u_l x^T M_l x \right] + x^T q  \leq
t.$$
By noting that the summation can be written from $l = 1, \hdots,
	K$, through the application of $\max_{\| u \|_\infty \leq 1} \eta^Tu = \|\eta\|_1, $
it follows that
$$x^T M_0 x +  \sum_{l=1}^L |x^T M_l x| +  x^T q \leq
t.$$
Since $M_l \succeq 0$ for every $l$, $|x^T M_l x| = x^T M_l x$.
{Consequently, the robust counterpart of \eqref{U1defM} can be stated as
	the convex constraint:
$$x^T \left( M_0  +  \sum_{l=1}^L  M_l \right) x + x^T q \leq
t.$$
}
Similarly, the constraint $M(u)x + q \geq 0, \quad \forall u \in
	\Uscr$ can be reformulated as follows:
\begin{align*}
 \quad &	M(u) x + q \geq 0, \quad  \forall u \in \Uscr \\
	\Leftrightarrow \quad  & M_0 x + \sum_{l=1}^L u_l M_lx + q \geq 0, \quad  \forall u \in \Uscr 
	\Leftrightarrow \quad   M_0 x + \min_{u \in \Uscr_{\infty}} \left[ \sum_{l=1}^L u_l M_lx \right] + q \geq 0,  \\
\Leftrightarrow \quad  & M_0 x  -\max_{u \in \Uscr_{\infty}} \left[
\sum_{l=1}^L u_l \left[- M_lx \right]  \right] + q \geq 0  
\Leftrightarrow \quad   M_0 x - 
\sum_{l=1}^L | -M_lx| + q \geq 0,\\
\Leftrightarrow \quad  & M_0 x - 
\sum_{l=1}^L | M_lx|  + q \geq 0.
\end{align*}
Through the addition of variables, $z_1, \hdots, z_{L}$, the
	resulting robust counterpart can then be stated as the following
		convex QP:
\begin{align}
\begin{array}{rrl}
\mbox{min} &  x^T (M_0 + \sum_{l=1}^L  M_l) x + x^T q\\
\mbox{subject to} & M_0 x + q - \sum_{l=1}^{L} z_l  & \geq 0,\\
&  z_l \geq M_l x & \geq -z_l, \qquad l = 1, \hdots, L \\
& z_l, x & \geq 0. \qquad \quad l = 1, \hdots, L
\end{array}
\end{align}

\item[(b)] $\Uscr:=\Uscr_1$ {As in (a), we begin
	by determining the robust counterpart of \eqref{U1defM}:}
$$x^T M_0 x + \max_{u \in \Uscr_1} \left[\sum_{l=1}^L u_l x^T M_l x \right] + x^T q  \leq
t.$$
{By noting that the summation can be written from $l = 1, \hdots,
	L$, through the application of $\max_{\| u \|_1\leq 1} \eta^Tu =
		\|\eta\|_{\infty}, $
it follows  that
$$x^T M_0 x +  \max_{l \in \{1, \hdots, L\}} |x^T M_l x|+ x^T q \leq
t.$$
But $x^TM_lx \geq 0$ for all $x$ and $l = 1, \hdots, L$  implying that
this constraint can be rewritten as follows:
$$x^T M_0 x + \max_{l \in \{1, \hdots, L\}} \{x^T M_l x \} + x^T q \leq
t.$$
The max. function can be eliminated by replacing each constraint by a
finite collection:
$$\begin{aligned}
	x^T M_0 x +   x^T M_l x + x^T q &  \leq t, \quad 
		l = 1, \hdots, L
	\end{aligned}$$
}
{Similarly, the semi-infinite constraint $M(u)x + q \geq 0, \quad \forall u \in
	\Uscr$ can be reformulated as follows:
\begin{align*}
 \quad &	M(u) x + q \geq 0, \quad \forall u \in \Uscr 
	\Leftrightarrow \quad   M_0 x + \sum_{l=1}^L u_l M_lx + q \geq 0, \quad \forall u \in \Uscr \\
	\Leftrightarrow \quad  & M_0 x + \min_{u \in \Uscr_{1}} \left[ \sum_{l=1}^L u_l M_lx \right] + q \geq 0
\Leftrightarrow \quad   M_0 x  -\max_{u \in \Uscr_{1}} \left[
\sum_{l=1}^L u_l \left[- M_lx \right] \right] + q \geq 0,  \\
\Leftrightarrow \quad  & M_0 x - 
\max_{l \in \{1, \hdots, L\}} \left[| -M_lx|\right]+ q \geq 0
\Leftrightarrow  \quad  \left\{\begin{aligned}
 		\quad  & M_0 x - z + q \geq 0, \\
 \mbox{ } \quad &   z  \geq M_lx  \geq -z, \quad  \forall l = 1, \hdots, L\end{aligned}\right\}. 
\end{align*}
Consequently, the TRC is given by the
	following:
$$
\begin{array}{rrl}
\mbox{min} & t\\
\mbox{subject to} &  x^T (M_0 + M_l) x + x^T q & \leq t, \qquad  \forall l\\
& M_0 x + q - {z} & \geq 0, \\
& z \geq M_lx & \geq -z, \qquad \forall l \\
& x,z & \geq 0.
\end{array}
$$
}
\item[(c)] $\Uscr:=\Uscr_2$: {We first
consider constraint \eqref{U1defM} which can be equivalently stated as
	follows: 
\begin{align*}
 &\qquad  x^T M_0 x + \max_{u \in \Uscr_2} \left[\sum_{l=1}^L u_l x^T M_l x \right] + x^T q  \leq
t \\
\Leftrightarrow & \qquad x^T M_0 x + \sqrt{ \sum_{l=1}^L ( x^T M_l x)^2  } + x^T q  \leq
t. 
\end{align*}
Similarly, the constraint $M(u) x+ q \geq 0$ for every $u \in \Uscr
	\equiv \Uscr_2$
can be reformulated as follows:
\begin{align*}
 \quad &	M(u) x + q \geq 0, \quad & \forall u \in \Uscr \\
		\Leftrightarrow \quad  & M_0 x + \min_{u \in \Uscr_{2}} \left[ \sum_{l=1}^L u_l M_lx \right] + q \geq 0,  \\
\Leftrightarrow \quad  & M_0 x  -\max_{u \in \Uscr_{2}} \left[
\sum_{l=1}^L u_l \left[- M_lx \right] \right] + q \geq 0,  \\
\Leftrightarrow \quad  & [M_0]_{i\bullet} x - \sqrt{ \sum_{l=1}^L\left[|
	[M_l]_{i\bullet}x|\right]^2}+ q_i \geq 0, \qquad \forall i. 
\end{align*}
Consequently, the robust counterpart of \eqref{unlcpM} can be stated as:
\begin{align} \label{nonmono1}
\begin{array}{rrl}
\min  & t \\
\st & x^T M_0 x + \sqrt{ \sum_{l=1}^L ( x^T M_l x)^2  } + x^T q &  \leq
t \\
& [M_0]_{i\bullet} x - \sqrt{ \sum_{l=1}^L\left[|
	[M_l]_{i\bullet}x|\right]^2}+ q_i & \geq 0, \qquad \forall i \\
& x & \geq 0.
\end{array}
\end{align}}
{By examining the second derivative of	$f(x)$  defined as
		$$ f(x) \triangleq  \sqrt{\sum_{l=1}^L (x^T M_l x)^2 },$$ 
	it can be  concluded that $f$ is a convex function. This
		result indicates that the the left hand side of the first
		constraint in \eqref{nonmono1} is a convex function, implying
		that the resulting feasible region is convex. The $n$ remaining
					   inequalities are in the form of second-order cone
					   constraints and are therefore tractable convex
					   constraints. It follows that \eqref{nonmono1}
		is a convex program.}
\end{itemize}
\end{proof}
To get a geometric understanding of the prior proposition, we consider the following example.

\noindent {\em Example:} {Consider the case when $M(u)$ and $q$ are defined as follows:
$$M(u) = u \pmat{1& 0\\0 & 2 },  q =
\pmat{2\\2}, \mbox{ and } \Uscr = \{u \mid -1 \leq u \leq 1\}.$$ It can be observed that the constraint: $x^T(M(u)x+q) \leq t$ is nonconvex when $u < 0$. Note that this constraint can be rewritten as $ u(x_1^2 + 2x_2^2)+2x_1+2x_2
\leq t$ and we defined $R_u$ as follows:
 $$ R_u \triangleq \left\{(x_1,x_2,t) \mid u(x_1^2 + 2x_2^2) + 2x_1+2x_2
\leq t\right\}.$$ Then $R_{-1}$ denotes the region above the surface shown in Fig.
\ref{hiddenconvexity} labeled $u=-1$, clearly a nonconvex set. Likewise, the
feasible regions $R_1, R_0$ represent the regions above the surfaces presented in
Fig. \ref{hiddenconvexity} labeled $u=1$ and $u=0$, respectively. Though the
set $R_{-1}$ is nonconvex and appears to make the program challenging to solve, a better understanding emerges when we consider the intersection of $R_u$ over $u$, as given by 
 $$R \triangleq \bigcap_{-1 \leq u \leq 1} R_u.$$ It can be seen that
$R_1 = R$. The figure on the left in Fig.
\ref{hiddenconvexity}  hints as to why this holds. The three surfaces
intersect at a single point, namely $(0,0)$ and the surface with greater
index $u$ stays above that with the lower index. This implies that $R_1 \subseteq
R_0 \subseteq R_{-1}$. Actually, $R_u$ is monotone in $u$ in that if $u_1 \leq
u_2$, then $R_{u_2} \subseteq R_{u_1}$. When considering such constraints in higher dimensions, similar behavior emerges. {Finally, there have been prior observations regarding the presence of hidden convexity in nonconvex programs (cf.~\cite{bental96hidden})}.}
\begin{figure}[thbp!]
\vspace{-0.1in}
\begin{center}
\includegraphics[width=7in]{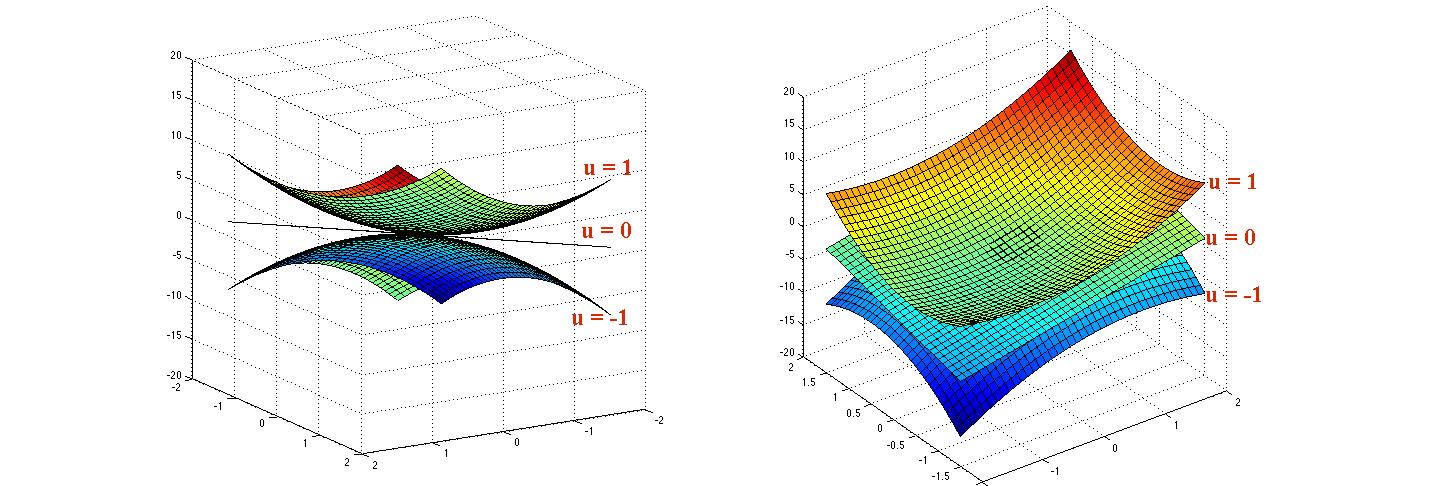}
\end{center}
\caption{Hidden convexity in two and three dimensions}
\label{hiddenconvexity}
\end{figure}

\subsection{A general tractability result}
In the prior subsections, we have provided a tractability statement when $M(u)$ and $q(u)$ are
		defined as per
		$$ M(u) = M_0 + \sum_{l=1}^L u_l M_l \mbox{ and } q(u) = q_0 +
		\sum_{l=1}^L u_l q_l. $$ A natural question is whether a more general tractability statement is available. We address precisely such a question by assuming that $(M,q) \in \wscr$, where $\wscr$ represents a more general uncertainty set. Note that in the settings considered in the prior subsection, $\wscr$ takes the form given by the following:	$$ \Wscr \triangleq \left\{ (M(u),q(u)): M(u) = M_0 + \sum_{l=1}^L u_l M_l \mbox{ and } q(u) = q_0 +
		\sum_{l=1}^L u_l q_l, u \in \Uscr\right\}. $$
We address the tractability question by considering the related separation problem. This requires a crucial result that relates the tractability of the separation and the optimization problems, both of which are defined next over a compact convex set.
{ \bi
\item[({\bf OPT})] {\bf Optimization problem}: Given a vector $c \in \mathbb{Q}^n$ and
a number $\epsilon > 0$, find a vector $y \in \mathbb{Q}^n$ such that
$y$ is an $\epsilon$-feasible and $\epsilon$-optimal solution of the
	problem:
	$$ \max_{x \in \Zcr} \ c^Tx. $$
 If $d(y, \Zcr)$ denotes the Euclidean distance of $y$ from $\Zcr$,
	{then} $y$ satisfies the following: 
$$d(y,\mathcal{Z}) \leq \epsilon \quad \mbox{($\epsilon$-feasibility)} \
\mbox{ and } \  c^T x \leq c^Ty +
\epsilon, \mbox{ for all } x \in \Zcr \quad \mbox{($\epsilon$-optimality)}$$
\item[({\bf SEP})] {\bf Separation problem}: Given a vector $y \in \Qb$ and a
parameter $\epsilon > 0$,  one of the following {may be concluded}:
\be
\item [(i)] assert that $d(y, \Zcr) \leq \epsilon$; 
\item [(ii)] find a vector $c \in \Qb ^n$ such that $\|c\| \geq 1$ and for every $x\in \Zcr$, $c^Tx \leq c^Ty + \epsilon$.
\ee
\ei}
Then Grotschel et al.~\cite{grotschel1981ellipsoid} showed that these
two problems are equivalent through the following result.
\begin{lemma}[Grotschel et
al.~\cite{grotschel1981ellipsoid}]\label{grotschel-lemma}
{Suppose $\Zcr$ is bounded with a nonempty relative interior.} Then
there is a polynomial algorithm\footnote{The algorithms are polynomial with input length $n +
\log(\frac{1}{\epsilon})$.} for solving the separation problem for
$\Zcr$ if and only if there is a polynomial algorithm for solving the optimization problem for $\Zcr$.
\end{lemma}

In the context of
obtaining a robust solution to an LCP, the optimization problem of interest is
given by the following: 
\begin{equation}\label{uLCPLobj}
\begin{array}{rrl}
\min \qquad \qquad d^Tz & \\
\st & z \in {\cal Z} \triangleq \left\{ z = (x;t) : \begin{aligned} x^T(Mx + q) & \leq t,\qquad  \forall (M,q) \in \wscr\\
 Mx + q & \geq 0, \qquad \forall (M,q) \in \wscr\\
 x & \geq 0
	\end{aligned} \right\},
\end{array}
\end{equation}
where $d \triangleq (0;1)$.
\begin{theorem}[{\bf Tractability of RC to uncertain LCPs with general uncertainty sets}] 
Suppose {there exists} a polynomial algorithm to solve the separation
problem over the compact convex uncertainty set $\mathcal{W}$ and the following assumptions hold:
\bi
\item[(a)] $M \succeq 0, \forall (M,q) \in \wscr$.
\item[(b)] If  $M_{i,\bullet}$ denotes the $i$th row of $M$, then
for $i = 1, \hdots, n$, the following holds:
	$$\|
M_{i,\bullet}\|_2 \geq \lambda \in \Qb_{++}, \qquad \forall (M,q) \in \wscr.$$
\item[(c)] Both $\wscr$ and $\mathcal{Z}$ are bounded sets, each having a {nonempty  relative
	interior.}
\ei
Then the problem 
\eqref{uLCPLobj} may be solved by a polynomial algorithm. \end{theorem}
\begin{proof}
It suffices to prove is that separation problem over the set $\mathcal{Z}$ can be solved in polynomial time since the cost vector $d$ in \eqref{uLCPLobj} has rational entries; specifically, $d = (0, \hdots, 0;1)^T$.

First notice that the assumption (a) {implies} that $\Zcr$ is convex
since it is defined as the intersection of an infinite collection of convex
sets. Therefore, the separation problem over $\Zcr$ is defined in a general form by ({\bf
		SEP}). {We proceed to show that this separation problem may be solved in polynomial time. 
  It suffices to show that either (i) or (ii) in ({\bf SEP}) can be shown to hold in
polynomial time for a given point $z = (x;t) \in \Qb^{n+1}$ and a positive
parameter $\epsilon$.} 

{Consider the
feasibility of $z$ with respect to $\cal Z$. Given a vector $z = (x;t)$ where $x$ is nonnegative\footnote{Note that if $x$ has any negative component, a separating hyperplane can be constructed with relative ease and we ignore this possibility without any loss in generality.}, it
suffices to examine whether $z$ is feasible with respect to the following:
\begin{align}
	x^T(Mx + q) \leq t, \quad \forall (M,q) \in \wscr\quad  & \equiv \max_{ (M,q) \in \wscr} x^T(Mx + q) \leq t \\
\mbox{For $i = 1, \hdots, n, \quad $}   	(M_{i,\bullet} x + q_i) \geq 0, \quad \forall (M,q) \in \wscr \quad & \equiv \min_{(M,q) \in \wscr} (M_{i,\bullet} x + q_{i}) \geq 0,  
\end{align}
 In effect, feasibility can be ascertained if the following problems may be solved in polynomial time:
$${\bf (1).} \ \left\{ \max_{(M,q) \in \wscr} \ x^TMx + q^Tx
\right\}
\mbox{ and } {\bf (2i).} \ 
\left\{ \min_{(M,q) \in \wscr} {M}_{i,\bullet}x + q_i \right\}, \quad i = 1, \hdots, n.$$
} 

	{Given an $(x;t) \in \Qb^{n+1}$, problem (1) is linear in $(\mbox{vec}(M);q)$ and can be recast as a problem of maximizing a linear objective over a convex set. Furthermore, an $\epsilon_1$-optimal solution of (1) will be denoted by $(M_0,q_0)$ where $\epsilon_1$ will be subsequently defined. Similarly, problem (2i) is a convex program with an objective that is linear in $(M_{i,\bullet}^T;q_i)$ and thereby linear in $(M;q)$. An $\epsilon_2$-optimal solution of (2i) will be denoted by $(M^i,q^i)$, where $\epsilon_2$ will be subsequently defined and the $i$th row of $M^i$ and $q^i$ will be collectively denoted by $(a_i,b_i).$  Since     the coefficients of both cost functions are rational by assumption, {\bf (1)}
	and {\bf (2i)} 
	for $i = 1,\hdots, n$ can be solved
		in polynomial time upto any precision $\epsilon_1 > 0$ and
		$\epsilon_2 > 0$, respectively and provide rational solutions.}

By the boundedness of $\Zcr$, we may assume that $\|y\| \leq \mcr$ for all $y \in \Zcr.$ 
 {Prior to proceeding, we quantify the precision levels $\epsilon_1$ and $\epsilon_2$}. Given an $\epsilon$, suppose $\epsilon_1$ and $\epsilon_2$ are defined as follows:
$$\epsilon_1 \triangleq \min \{\epsilon/(3\|x\|^2 + 3\|x\|+1), \epsilon/(3(\mcr+\|x\|)(2\|x\| + 1))\} \text{ and } \epsilon_2 \triangleq \min \{\lambda \epsilon/(6\mcr), \lambda \epsilon/6, \lambda/2 \}.$$ 

Suppose $\wscr_i \triangleq \{(M_{i,\bullet},q_i): (M,q) \in \wscr\}.$
Since $d((M_0,q_0), \wscr) \leq \epsilon_1$, $d((a_i,b_i), \wscr_i) \leq \epsilon_2$, and $\wscr$ is compact, we may find $(M_1,q_1) \in \wscr$ and $(\bar a_i, \bar b_i) \in \wscr_i$ such that $d((M_0,q_0),(M_1,q_1))\leq \epsilon_1$ and $d((a_i,b_i),(\bar a_i, \bar b_i)) \leq \epsilon_2$. Hence,
\begin{align}
|x^T(M_0x + q_0) - x^T(M_1x + q_1)| & = |x^T((M_0-M_1)x + (q_0-q_1))|\cr
		& \leq \|x \|\|(M_0-M_1)x + (q_0-q_1)\| \cr
& \leq \|x \|(\|M_0-M_1\|\|x\| + \|q_0-q_1\|) \cr
& \leq \|x\|(\|x\|\epsilon_1+\epsilon_1)\cr
& \leq (\|x\|^2 + \|x\|)\epsilon/(3\|x\|^2 + 3\|x\|+1) < \epsilon/3. \label{eq1}
\end{align}
Furthermore, we have the following:
\begin{align}
\forall y \in \Zcr,  \text{we have }
& \quad | ((M_0+M_0^T)x + q_0)^T(y-x) - ((M_1+M_1^T)x + q_1)^T(y-x) | \cr
& \leq \|y-x\| \|(M_0-M_1+M_0^T-M_1^T)x + q_0-q_1\| \cr
& \leq \|y-x\|(2\|M_0-M_1\|\|x\|+\|q_0-q_1\|) \cr
& \leq \|y-x\|(2\|x\|+1)\epsilon_1\cr
&  \leq  (\mcr + \|x\|)(2\|x\|+1)\epsilon_1 \leq \epsilon/3. \label{eq23}
\end{align}
Similarly, we may bound the difference between $(a_i,b_i)$ and $(\bar a_i,\bar b_i)$ for all $i$:
\begin{align}
\mbox{ for $i = 1, \hdots, n$},  \text{we have }  
|b_i/\lambda - \bar b_i/\lambda | & \leq \epsilon_2/\lambda \leq \epsilon/6 \label{eq21} \\
\mbox{ and for $i = 1, \hdots, n$ and for all $y \in \Zcr$},  \text{we have }  
 |(a_i/\lambda-\bar a_i/\lambda)^Ty| & \leq \|a_i-\bar a_i\|\|y\|/\lambda \cr
			& \leq \|a_i-\bar a_i\|\mcr/\lambda \cr
			& \leq \mcr\epsilon_2/\lambda \cr
& \leq \epsilon/6. \label{eq22}
\end{align}
In addition, {since problem (2i) is solved to a precision of $\epsilon_2$, we have that the following holds: 
\begin{align*}
 \mid \| a_i \| -  \| \bar a_i \| \mid & \leq \epsilon_2 
   \leq \lambda/2. 
\end{align*}
Then we may conclude the following:
\begin{align}
 \Rightarrow & (\| a_i \| -  \| \bar a_i \|)/\lambda  \geq -1/2 \cr
\Rightarrow & \| a_i \|/\lambda \geq \| \bar a_i \| /\lambda -1/2 \geq 1-1/2 = 1/2,  \label{eq24}
\end{align}
where the last inequality is a consequence of the assumption $\|\bar a_i \| \geq \lambda.$
}

The final part of the proof is reliant on considering three possibilties that emerge based on obtaining the solutions to problems (1) and (2i) for $i = 1, \hdots, n$.
\bi
\item[(1)] Suppose $x^TM_0x + q_0^Tx \leq t-\epsilon_1 $ and $ a_ix + b_i \geq \epsilon_2$ for $i = 1,\hdots, n$.  Since $(M_0,q_0)$ is an $\epsilon_1$-optimal solution, we have that 
$$\max_{(M,q) \in \wscr} (x^TMx + q^Tx) \leq x^TM_0x + q_0^Tx + \epsilon_1 \leq t- \epsilon_1 + \epsilon_1 \leq t,$$ where the second inequality follows by our assumption on $(M_0,q_0)$. Similarly, we have that $(M_{i,\bullet},q_i)$ is an $\epsilon_2$-optimal solution implying that  $$\min_{(M,q) \in \wscr} ( M_{i,\bullet} x + q_i ) \geq a_ix + b_i - \epsilon_2 \geq \epsilon_2 - \epsilon_2 =0,$$
where the second inequality again follows by assumption. It can be concluded  that $(x;t) \in \Zcr$ and (i) of ({\bf SEP}) has been concluded in polynomial time.
\item[(2)] Suppose $x^TM_0x + q_0^Tx > t-\epsilon_1 $ and $ a_ix + b_i \geq \epsilon_2$ for $i = 1,\hdots, n$.
Then we may construct a vector $c$ defined as  $c \triangleq [(M_0 + M_0^T)x + q_0;-1]$. Consequently, $\|c\| \geq 1$. Furthermore, since $M_0$ and $q_0$ have rational entries by assumption, it follows that $c \in \Qb^{n+1}$. Furthermore, 
$\forall (y;\tau)$ feasible, we have the following sequence of inequalities:
\begin{align*} \tau & \geq y^T(M_1y + q_1) 
		    \geq x^T(M_1x + q_1) + [(M_1 + M_1^T)x + q_1]^T(y-x),
\end{align*}
where the second inequality follows from the convexity of $y^T(M_1 y + q_1)$. It follows that \begin{align*} 
& \quad x^T(M_1x + q_1) + [(M_1 + M_1^T)x + q_1]^T(y-x) \\
	& \geq x^T(M_0x + q_0) - \epsilon/3 + [(M_1 + M_1^T)x + q_1]^T(y-x) \\
			&  > t -\epsilon_1 - \epsilon/3 + [(M_0 + M_0^T)x + q_0]^T(y-x)- \epsilon/3\\
		&  = t -\epsilon+ [(M_0 + M_0^T)x + q_0]^T(y-x).
\end{align*} 
where the first inequality follows from \eqref{eq1}, the second inequality follows by assumption on $x^TM_0x + q_0^Tx$ and by invoking \eqref{eq23}. Therefore we have that  
\begin{align*}
		c^Tz & = [(M_0 + M_0^T)x + q_0;-1]^T(x;t) + \epsilon\\
	&  \geq [(M_0 + M_0^T)x + q_0;-1]^T(y; \tau), \quad \forall (y;\tau) \in \Zcr.
\end{align*}
Consequently, we have concluded (ii) of ({\bf SEP}) in polynomial time. 
\item[(3)] Suppose there exists an $i \in \{1, \hdots, n\}$ such that $ a_i^Tx + b_i < \epsilon_2$. Then let $c$ be defined as 
$c \triangleq -(2a_i/\lambda; 0)$. Then from \eqref{eq24}, $\| c \| = \|2a_i\| /\lambda \geq 1$ and $c\in \Qb^{n+1}$ since $a_i \in \Qb^n, \lambda \in \Qb$. Thus
 \begin{align*}
		\forall (y;\tau) \in \Zcr, c^T(y;\tau) & = -2a_i^Ty/\lambda \\
			& \leq -2\bar a_i^Ty/\lambda + 2\epsilon/6 \\
	&  \leq \bar 2b_i/\lambda + \epsilon/3\\
	&  \leq 2b_i/\lambda + 2\epsilon/6 + \epsilon/3 \\
	& \leq -2a_i^Tx/\lambda + 2\epsilon_2/\lambda + 2\epsilon/3\\
	&  \leq -2a_i^Tx/\lambda + \epsilon = c^T(x;t) + \epsilon,
\end{align*}
where the first inequality follows from \eqref{eq22}, the second inequality follows from the feasibility of $(\bar a_i,\bar b_i)$, the third inequality is a consequence of \eqref{eq21}, the fourth inequality is a consequence of the assumption that $a_i^Tx + b_i < \epsilon_2$, and the last inequality is a result of invoking the definition of $\epsilon_2$. Therefore, we have concluded (ii) of ({\bf SEP}) in polynomial time.  
\ei 
Based on (1), (2), and (3), we note that the separation problem ({\bf SEP}) can be solved in polynomial time. Consequently, by Lemma ~\ref{grotschel-lemma}, there is a polynomial time algorithm for solving the optimization problem over $\Zcr$.
\end{proof}

{\bf Remark:} Through the above proposition, we establish a connection between the
tractability of the uncertainty set and the tractability of the robust
counterpart of uncertain LCP, further generalizing our findings from the prior subsections.

\section{General uncertain non-monotone LCPs}\label{sec:IV}
{In this section, we consider non-monotone uncertain LCPs in more
	general settings where tractable robust counterparts are unavailable.
		Instead, we examine when such problems result in finite dimensional nonconvex
		programs.  {We assume that $M(u)$ takes a form given by $M(u) = M_0 +
\sum_{l=1}^L u_l M_l$ and $q(u) = q_0 + \sum_{l=1}^L u_l q_l$ where $u \in
\mathcal{U}$, $(M_0,q_0)$ denote the nominal values on $M$ and $q$, and
$(M_l,q_l)_{l=1}^L$ are the basic shifts, while $\mathcal{U}$ represents a
given perturbation set. This model of utilizing nominal values and shifts has been relatively standard in the field of robust optimization (cf.~
\cite{bental09robust} and the references therein). } {In
	Section~\ref{sec:40}, we see that in the more general setting, the
		RC is already intractable {to resolve} when the perturbation
		set is {of} dimension one, {demonstrating the hardness of
			the problem}}. In Section~\ref{sec:41}, we discuss two avenues via which we may obtain nonconvex robust counterparts. While
		stationary points of  such
		problems can be obtained by nonlinear
		programming solvers, global solutions require branching-based
		schemes. In Section~\ref{sec:42}, inspired by recent research
		{by}
		Fampa et al.~\cite{fampa2013global}, we present a technique for
		obtaining global solutions of a nonconvex quadratically
		constrained quadratic program.}

\subsection{NP-hardness of the problem}\label{sec:40}
{Throughout this section, we define $M(u)$ as follows:
\begin{align}\label{defM4}
\begin{aligned}
M(u) \triangleq M_0 + \sum_{l=1}^L u_l M_l, q(u) \triangleq q_0 + \sum_{l=1}^L u_l q_l, u \in \Uscr.
\end{aligned}
\end{align}
without any assumption on $M_0$ or $M_l, l = 1,\hdots, L$. Next, we prove that even the RC of a simple LCP$(M(u),q(u))$, where $M(u) = M_0 + uM_1, q(u) = q_0 + uq_1, u \in \Uscr = [0,1]$ is NP-hard. 
	\begin{lemma}
Consider an uncertain linear complementarity problem LCP$(M(u),q(u))$
where  $M(u) = M_0 + uM_1, q(u) = q_0 + uq_1, u \in \Uscr = [0,1]$ and
$M_0$ and $M_1$ lie in $\Real^{n \times n}$. Then
the robust counterpart of this problem is NP-hard where the robust
counterpart is defined as follows:
\begin{align} \tag{RC$_{\rm LCP}$}
\begin{array}{rrl}
\min & {\displaystyle \max_{u \in \Uscr}} \ x^T(M(u)x + q(u)&\\
\st & M(u)x + q(u) & \geq 0, \qquad \forall u \in \Uscr \\
& x & \geq 0.
\end{array}
\end{align}
\end{lemma}
\begin{proof}
First we write the RC, given by (RC$_{LCP}$), as follows:
$$
\begin{array}{rrl}
\min & {\displaystyle \max_{u \in \Uscr}} \ x^T(M(u)x + q(u)&\\
\st & M(u)x + q(u) & \geq 0, \qquad \forall u \in \Uscr \qquad (\mbox{RC}_{\mathrm{LCP}})\\
& x & \geq 0.
\end{array}
$$
Let $M_0 = 0$ and $q_0 = 0$, then 
$$\max_{u \in \Uscr} \{x^T(M(u)x + q(u))\} = \max_{u \in [0,1]} \{ x^T(uM_1x + uq_1) \} =  \max \{0, x^T(M_1x + q_1)\}.$$ Furthermore, we have that $$M(u)x + q(u) \geq 0, \forall u \in [0,1] \quad \Leftrightarrow \quad M_1x + q_1 \geq 0.$$ As a consequence, $\max \{0, x^T(M_1x + q_1)\} = x^T(M_1x + q_1)$ for any nonnegative vector $x$ and the (RC$_{\mathrm{LCP}}$) is equivalent to the following:
$$
\begin{array}{rrl}
\min & x^T(M_1x + q)&\\
\st & \quad M_1x + q_1 & \geq 0,\\
& x & \geq 0,
\end{array}
$$ which is essentially the problem LCP$(M_1,q_1)$. But LCP$(M_1,q_1)$
is NP-complete since the equality-constrained $0$-$1$ knapsack problem can be reduced to 
LCP$(M,q)$ in polynomial time~\cite{chung1989np}. Consequently,
obtaining a robust solution to an uncertain LCP with general matrices
is NP-complete. \end{proof} 

\textbf{Remark:} If we assume that the feasible region $\{x \mid M_1x + q_1 \geq 0, x \geq 0\}$ is bounded, the LCP$(M_1,q_1)$ is still NP-complete. The reason is that equality-constrained 0-1 knapsack problem can be reduced to an LCP$(M,q)$ such that $\{x \mid Mx + q \geq 0, x \geq 0\}$ is bounded. Please refer to \cite{chung1989np} for details.
}

\subsection{Nonconvex robust counterparts}\label{sec:41}
In Section~\ref{sec:32}, we showed that the RC of uncertain nonmonotone
LCP may be tractable under some assumptions. However, in general, this
is not the case, {particularly when $M(u)$ and $q(u)$ are defined on
	the same (or related) uncertainty sets.} In this circumstance, we
	can show that the RC may still be reformulated as a finite dimensional mathematical
	program.

\begin{proposition}[For non-monotone uLCP$(M(u),q(u))$]\label{prop-nmon-nconv} 
Suppose $M(u)$ and $q(u)$ are defined as \eqref{defM4} where $\mathcal{U}$ is
defined as $\Uscr_1, \Uscr_2,$ or $\Uscr_\infty$, ($\Uscr_\infty$ is defined in
\eqref{qU}). Then \eqref{unlcp2} may be written as a finite dimensional mathematical program.
\end{proposition}
\begin{proof}
{We introduce an artificial variable $w$ into a finite set of nonconvex quadratic equality constraints:
\begin{align}
\begin{aligned} 
w_l & = x^T M_l x + q_l^T x, & l = 1,\hdots, L,
\end{aligned} \label{nonmono2}
\end{align}
}
Then the quadratic constraint in \eqref{unlcp2} can be written as
follows:
\begin{align}\label{nonmono3}
x^T M_0 x + q_0^T x + u^T w \leq t,\qquad \forall u \in \mathcal{U}
\end{align}
This semi-infinite constraint can be equivalently written as follows:
\begin{align}\label{nonmono4}
x^T M_0 x + q_0^T x + \max_{u \in \Uscr} u^T w \leq t. 
\end{align}
But, $\displaystyle \max_{u \in \Uscr} u^T w$ is given by $\|w\|_1$,
	$\|w\|_{\infty}$ or $\|w\|_2$ if $\Uscr$ is given by $\Uscr_{\infty},
	\Uscr_{1}$ or $\Uscr_2$, respectively. Consequently,
	\eqref{nonmono3} may always be recast in a tractable fashion.
Similarly, consider the second constraint in \eqref{unlcp2}:
\begin{align}
M_0 x + q_0 + \sum_{l=1}^L u_l (M_l x + q_l)
 \geq 0,\quad  \forall
 u \in \mathcal{U}.\label{sec_cons}
\end{align}
Based on Prop.~\ref{prop_nmon_lcp}, the constraint \eqref{sec_cons} can also be
reformulated in a tractable fashion when $\Uscr = \Uscr_1, \Uscr_2$ or $\Uscr_{\infty}$. 
We demonstrate this in detail {when} $\Uscr = \Uscr_\infty$ and
{omit the proofs when} $\Uscr= \Uscr_1$ and $\Uscr_2$:
Note that when $\Uscr = \Uscr_\infty$, \eqref{nonmono4} can be rewritten as 
\begin{align}\label{nonmono6}
x^T M_0 x + q_0^T x + \|w\|_1 \leq t. 
\end{align}
{Furthermore \eqref{sec_cons} is equivalent to
\begin{align*}
(M_0x)_i + (q_0)_i - \| Z^i \|_1 \geq 0,
(Z^i)_l = (M_lx + q_l)_i, \quad
\forall l = 1, \hdots, L, \forall i = 1, \hdots, n.
\end{align*}}
Through the addition of variables, $\tau_1,\hdots, \tau_L, z_1, \hdots, z_L$, the
	resulting robust counterpart when $\Uscr = \Uscr_\infty$ can then be stated as the following optimization problem: 
$$\begin{array}{rrlr}
{\displaystyle \min_{x \geq 0, t}} & \qquad t &&\\
\mbox{subject to} & \qquad
  -\tau_l \leq   x^T M_l x + q_l^T x & \leq  \tau_l,    & l = 1,\hdots, L \\ 
 & x^T M_0x + q_0^Tx +
 \sum_{i=1}^{L} \tau_i & \leq t,  &\\
& M_0 x + q_0 - \sum_{i=1}^{L} z_i  & \geq 0, & \\
&  z_l \geq M_l x + q_l & \geq -z_l, \qquad   & l = 1, \hdots, L. 
\end{array}
$$
This QCQP is nonconvex when $M_l, l = 1,\hdots,L$ are indefinite.
\end{proof}

We now provide a corollary of this result when $\Uscr:=\Uscr_c$.
\begin{corollary}[RC for non-monotone uLCP$(M(u),q(u))$ where
	$\Uscr = \Uscr_c$]
{Suppose $M(u)$ is defined as \eqref{defM4} and $q(u) = q_0 +
\sum_{i=1}^{L} u_i q_i$ where $u \in \mathcal{U}$ and $\mathcal{U}$ is
given by $\Uscr_c$, defined as \eqref{defUc}.
	Then	\eqref{unlcp} can be represented as a finite-dimensional mathematical program:
$$
\begin{array}{rrlr}
{\displaystyle \min_{x \geq 0}}& x^TM_0x + q_0^Tx+ p^Ty \\
\textrm{\em subject to} & (P^T y)_l + x^T M_lx + q_l^T x & = 0, &
\forall l = 1,\hdots,L, \\
& Q^T y & = 0,\\
& y & \in K^*,\\
& [M_0]_{i \bullet}x + [q_0]_i - p^T z_i & \geq 0, & \forall i = 1,\hdots, n \\
& [P_{\bullet l}]^T z_i & = [M_l]_{i\bullet}x + [q_l]_i, & \forall i = 1,\hdots, n, l = 1,\hdots,L\\
& Q^T z_i & =0, & \forall i = 1,\hdots, n\\
& z_i & \in K^*,& \forall i = 1,\hdots, n.
\end{array}
$$
}
\end{corollary}

\noindent {\bf Remark:} While stationary points of such problems may be computed through
standard nonlinear programming schemes such as globalized sequential
quadratic programming or interior point
methods~\cite{nocedal99numerical}, our interest lies in obtaining global
solutions of such problems. In the next subsection, we review several approaches for
obtaining global solutions to QCQPs.

\subsection{A branching scheme for resolving nonconvex
	QCQPs}\label{sec:42}
{Before presenting our scheme, we provide a brief review of global
	optimization schemes for resolving indefinite quadratic programs and
their quadratically constrained generalizations. Such a class of problems
has seen significant study~\cite{anstreicher2004sdp, kojima2000cones}.
In \cite{anstreicher2004sdp}, the authors  
	 combine reformulation-linearization-technique (RLT) with an SDP
	 relaxation to tackle QCQP. In \cite{kojima2000cones},  a general
	 framework is buit for solving such problems. 
	While branching schemes come in varied forms, Burer and
	Vandenbussche~\cite{burer2008finite} employ SDP relaxations for
	addressing indefinite quadratic programming. 
We consider a spatial branch-and-bound approach inspired by Fampa et al.
 \cite{fampa2013global} developed for nonconvex quadratic programs. This
 approach uses secant inequalities for deriving a relaxation {of the quadratic objective}. We extend
 this approach to quadratically constrained variants. We emphasize that
 the focus of this paper lies in extending standard robust optimization
 techniques to allow for accommodating uncertain linear complementarity
 problems. While a comprehensive study of branching schemes is beyond
 the scope of the current paper, we show that at least one of the
 approaches can be readily adapted to this context.}  

We begin by noting that the QCQP can be recast as an optimization problem with a linear objective and quadratic constraints, some of which may be nonconvex. We continue using the defnition \eqref{defM4} and illustrate the scheme for the case when $\Uscr := \Uscr^r_\infty = \{u \mid \| u \|_\infty \leq 1, u \geq 0 \}$ and qualify the relaxations and the bounds by using the superscript $\infty$.
Suppose $u \in \Uscr^r_\infty$. From
Prop~\ref{prop-nmon-nconv},  the optimization problem given by \eqref{unlcp2} may be reformulated as follows:
\begin{align}\label{nonconvexprog}
\begin{array}{rrl}
{\displaystyle \min_{x \geq 0, t}} & t \\
\st & x^T M_0 x + q_0 ^T x + \sum_{i=1}^{L} \max \{ x^T M_i x + q_i ^T x , 0 \} & \leq t,\\
& (M_0 x + q_0)_k + \sum_{i=1}^{L} \min \{(M_i x + q_i)_k ,0\} & \geq 0,\qquad \forall k = 1,\hdots,n.
\end{array}
\end{align}
{While the second row of the constraint part in \eqref{nonconvexprog} can be immediately written as:
\begin{align*}
M_0 x + q_0 - \sum_{i=1}^{L} z_i \geq 0, \quad
M_ix + q_i + z_i \geq 0, \quad
z_i \geq 0, \quad \forall i = 1,\hdots,L,
\end{align*}
the chief concern lies in the first constraint which can be
	decomposed into and $2L + 1$ constraints:
\begin{align}\label{nonconvex-cons}
x^T M_0 x + q_0 ^T x + \sum_{i=1}^{L} \tau_i &\leq t,
\mbox{ and } \tau_i \geq 0, x^T M_i x + q_i ^T x  \leq \tau_i, \qquad i = 1, \hdots, L.
\end{align}
	
\noindent {\bf Constructing a relaxation:} 	Akin to the approach employed in
	\cite{fampa2013global}, we use the eigenvalue decomposition of
	$ M_i$, defined as $$M_i = -\sum_{j =1}^{J_i} \lambda_{i,j} \nu_{i,j} \nu_{i,j}^T + \sum_{j = 1}^{K_i} \mu_{i,j} \eta_{i,j}\eta_{i,j}^T, $$ where $\lambda_{i,j} > 0, \mu_{i,j} > 0, \forall j, \forall i = 1,\hdots, L.$ Let $M_i^+ = \sum_{j = 1}^{K_i} \mu_{i,j} \eta_{i,j}\eta_{i,j}^T, \forall i = 0,\hdots, L.$ By defining $y_{i,j}$ as
		$y_{i,j} = \sqrt{\lambda_{i,j}}\nu_{i,j}^T
		x$, quadratic inequalities in \eqref{nonconvex-cons} may be rewritten as 
\begin{align} \label{nonconvex-cons-y}
-\sum_{j=1}^{J_0} y_{0,j}^2 + x^TM_0^+ x + q_0^T x + \sum_{i=1}^L\tau_i & \leq t, \\
-\sum_{j=1}^{J_i} y_{i,j}^2 + x^T M_i^+ x + q_i ^T x  & \leq \tau_i,\qquad  i = 1,\hdots, L.
\end{align}
Suppose $l_{i,j} \leq y_{i,j} \leq u_{i,j}$ for $j =
		1, \hdots, J_i$ and $i = 0,\hdots, L$. Then we may use a
		secant inequality for providing a relaxation to the
		\eqref{nonconvex-cons-y} in the form of the following:
\begin{align} \label{nc-cons-y-rel}
 -\sum_{j=1}^{J_0} \left( (y_{0,j} - l_{0,j}) \left( u_{0,j} + l_{0,j} \right) + l_{0,j}^2 \right) + x^TM_0^+x + q_0^T x + \sum_{i=1}^L\tau_i \leq t, \\
 -\sum_{j=1}^{J_i} \left( (y_{i,j} - l_{i,j}) \left( u_{i,j} + l_{i,j} \right) + l_{i,j}^2 \right) + x^TM_i^+x + q_i^T x  \leq
\tau_i,\qquad  i=1,\hdots,L. 
\end{align}
When $\Uscr:=\Uscr^r_{\infty}$, the resulting relaxed problem is denoted
	by (P$_{\infty}(l,u)$) and is defined as
follows:
\begin{align}\label{convexrelax}
\begin{aligned}
\begin{array}{rrll}
\min & t \\
&  -\sum_{j=1}^{J_0} \left( (y_{0,j} - l_{0,j}) \left( u_{0,j} + l_{0,j} \right) + l_{0,j}^2 \right) + x^TM_0^+x + q_0^T x + \sum_{i=1}^L\tau_i & \leq t,  \\
& \tau_i  & \geq 0, & i = 1, \hdots, L, \\
&  -\sum_{j=1}^{J_i} \left( (y_{i,j} - l_{i,j}) \left( u_{i,j} + l_{i,j} \right) + l_{i,j}^2 \right) + x^TM_i^+x + q_i^T x & \leq
\tau_i, &  i=1,\hdots,L, \\
& l_{i,j} \leq y_{i,j} & \leq u_{i,j}, & \forall j, i,\\
& y_{i,j} - \sqrt{\lambda_{i,j}} \nu_{i,j}^T x & = 0, & \forall j,
	i,\\
& M_0 x + q_0 - \sum_{i=1}^{L} z_i& \geq 0, & \\
& M_ix + q_i + z_i & \geq 0, & \forall i = 1,\hdots,L \\
& z_i & \geq 0, & \forall i = 1,\hdots,L \\
& x & \geq 0. & 
\end{array}
\end{aligned}
\tag{P$_{\infty}$($l,u$)}
\end{align}
\noindent {\bf Obtaining upper and lower bounds for $y_{i,j}$:}
{Crucial to this scheme is the need for obtaining upper and lower
	bounds on $y_{i,j}$ given by $l_{i,j}$ and $u_{i,j}$, respectively.
Consider the set $X_\infty$
\begin{align*} X_{\infty}\triangleq \left\{x \middle|\exists z_i \in \mathbb{R}^n, M_0 x + q_0 - \sum_{i=1}^{L} z_i \geq 0, \quad M_ix + q_i + z_i  \geq 0,\quad z_i  \geq 0, \quad \forall i = 1,\hdots,L, \quad x \geq 0 \right\}.
\end{align*}
The lower and upper bound for $y_{i,j}$ can then be obtained by solving the following set of linear programs:
\begin{align}\label{bound}
\begin{aligned}
\min \slash \max & \qquad \sqrt{\lambda_{i,j}} \nu_{i,j}^T x\\
\st &  \qquad x \in X_\infty.
\end{aligned} \tag{$l^{\infty}_{i,j}\slash u^{\infty}_{i,j}$}
\end{align}
}
Note that we assume these LPs are bounded and when the uncertainty set is either $\Uscr_1$ or $\Uscr_2$, the
relaxation and the upper/lower bounds have to be derived in an analogous
fashion. A formal outline of the branching scheme is provided in
Algorithm~\ref{algorithm:bb}.
}
\begin{algorithm}[H]
  \caption{Spatial branch and bound}
\label{algorithm:bb}
\begin{boxedminipage}{165mm}
    \begin{algorithmic}[1]
    \STATE \textbf{Init.i:} $k:=1$; $\texttt{terminate} := 0$; choose $\epsilon >0$;
	\STATE \textbf{Init.ii:} $M_i = -\sum_{j =1}^{J_i}
	\lambda_{i,j} \nu_{i,j} \nu_{i,j}^T + \sum_{j=1}^{K_i} \mu_{i,j} \eta_{i,j}\eta_{i,j}^T$ for $i = 0,\hdots, L$;
	\STATE \textbf{Init.iii:} For all $i,j$ compute $l_{i,j}^\infty$ and
	$u_{i,j}^\infty$; 
	\STATE  \textbf{Init.iv:} Let $P_1:= P_\infty(l^\infty,u^\infty)$;
	$(x_1^*,\tau_1^*,y_1^*,z_1^*,t_1^*) \in \argmin P_1$; 
	\STATE \textbf{Init.v:} Assign bounds: $\texttt{glb\_lb}: =
	t_1^*$; $\texttt{glb\_ub}:= x_1^{*T} M_0 x_1^* + q_0^Tx_1^* +
	\sum_{i=1}^{L} \max \{x_1^{*T} M_i x_1^* +
		q_i^T x_1^* ,0 \}$; 
	\STATE \textbf{Init.vi:} Update list: $\texttt{list}:=\{
	(P_1,\texttt{glb\_lb},\texttt{glb\_ub})\}$.
    \WHILE {$\texttt{terminate} == 0$}
        \STATE {\bf Branching index:} For $P_k$, choose index pair
		$(\overline{i},\overline{j}) := {\displaystyle \argmax_{(i,j)}}
		(u_{i,j}-l_{i,j})$; $\phi_{\overline{i},\overline{j}} =
		(u_{\overline{i},\overline{j}}+l_{\overline{i},\overline{j}})/2$;
		\STATE {\bf Update bounds:} 
		\begin{align*}
					\widehat u_{i,j} = \begin{cases}
									\phi_{\overline i,\overline j},
									& (i,j) = (\overline
											i,\overline j) \\
									u_{i,j}, & \textrm{otherwise}.
									\end{cases} \textrm{ and } 					\widehat l_{i,j} = \begin{cases}
									\phi_{\overline i,\overline j},
									& (i,j) = (\overline
											i,\overline j) \\
									l_{i,j}, & \textrm{otherwise}.
									\end{cases}
		\end{align*}
		\STATE {\bf Construct leaves:} $P_k^l := P_\infty(u,\widehat l)$; $P_k^u:= P_\infty(\widehat u,l)$. 
		\STATE {\bf Upper and lower bounds for P$_k^l$:}  $(x^l_*,\tau^l_*,y^l_*,z^l_*,t^l_*) \in \argmin (P_k^l)$;\\
	\begin{align*}
	\texttt{lb}^l:= t^l_*;\qquad \texttt{ub}^l: = (x^l_{*})^T M_0 x^l_*
		+ q_0^Tx^l_* + \sum_{i=1}^L \max \{(x^l_*)^{T} M_i x^l_* +
		q_i^Tx^l_* ,0\}
	\end{align*}
\STATE {\bf Upper and lower bounds for P$_k^u$:}
 $(x^u_*,\tau^u_*,y^u_*,z^u_*, t^u_*) \in \argmin (P_k^u)$;
	\begin{align*}
		\texttt{lb}^u:= t^u_*; \qquad \texttt{ub}^u: = (x^u_{*})^T M_0 x^u_*
		+ q_0^Tx^u_* + \sum_{i=1}^{L} \max \{ (x^u_*)^{T} M_i x^u_* +
		q_i^Tx^u_* \};
	\end{align*}
   \STATE {\bf Delete $P_k$ from \texttt{list}:}  $\texttt{list} := \texttt{list} \backslash P_k$;
        \STATE {\bf Append \texttt{list} by
			$(P_k^l,\texttt{lb}^l,\texttt{ub}^l)$:}  \\ If $t^l_* < 
		\texttt{glb\_ub}$, then $\texttt{list} := \texttt{list}
		\cup (P_k^l,\texttt{lb}^l,\texttt{ub}^l)$; If $\texttt{ub}^l < \texttt{glb\_ub}$, then $\texttt{glb\_ub} := \texttt{ub}^l$ and $P^{\ell} := P_k^l$;
        \STATE {\bf Append \texttt{list} by
			$(P_k^u,\texttt{lb}^u,\texttt{ub}^u)$:} \\ If $t^u_* <
		\texttt{glb\_ub}$, then $\texttt{list} := \texttt{list}
		\cup (P_k^u,\texttt{lb}^u,\texttt{ub}^u)$; If $\texttt{ub}^u < \texttt{glb\_ub}$, then $\texttt{glb\_ub} := \texttt{ub}^u$ and $P^{\ell} := P_k^u$;
        \STATE {\bf Termination test:} If $\texttt{glb\_lb} - \texttt{glb\_ub} < \epsilon$,
		then $\texttt{terminate} := 1$; Output $P^{\ell}$ and its solution.
        \STATE {\bf Choose $(P,\texttt{lb},\texttt{ub})$ from \texttt{list} such that the the associated lower bound \texttt{lb} is the smallest in the list and set the global lower bound $\texttt{glb\_lb} = \texttt{lb}$. Let $P_{k+1}:= P$.}

        \STATE $k:=k+1$;
    \ENDWHILE
   \end{algorithmic}
\end{boxedminipage}
\end{algorithm}

\section{Extensions to uncertain VIs and MPCCs}\label{sec:V}
In this section, we consider two key generalizations of the uncertain
monotone linear complementarity problem. In Section~\ref{sec:42}, we
extend this framework to the regime of affine variational inequality
problems over polyhedral sets. Next, we demonstrate how our framework
can address a subclass of stochastic mathematical programs with
equilibrium constraints (MPECs) (cf.~\cite{luo96mathematical}), given by
a stochastic quadratic program with (uncertain) linear complementarity
constraints. 
\subsection{Uncertain affine monotone polyhedral VIs} Two shortcomings immediately
come to the fore when considering the model \eqref{unlcp}:
\begin{enumerate}
\item[(i)] The
set $X$ is a cone;
\item[(ii)] The underlying set is deterministic in that it is uncorrupted by uncertainty.
\end{enumerate}
In this subsection, we show that examining uncertain polyhedral sets can
also be managed within the same framework. Specifically, we begin by considering an uncertain affine variational inequality problem over a polyhedral set of the form given by \eqref{unvi}
wherein $X(u)$ and $F(x,u)$ are defined as
\begin{align} \label{defFXu}
\begin{aligned}
	X(u)   \triangleq \{x: C(u)x \geq b(u), x \geq 0\}  \mbox{ and }
	F(x,u)  \triangleq M(u)x +
	q(u),
\end{aligned}
\end{align}
respectively. From ~\cite[Prop.~1.3.4]{facchinei02finite}, $x$ solves VI$(X(u),F(u))$ if and only if 
there exists a vector $\lambda \in \Real^m$ such that 
\begin{align}
\label{pavi}
\begin{aligned}
	0 \leq x & \perp M(u) x - C(u)^T\lambda + q(u) \geq 0 \\
	0 \leq \lambda & \perp C(u) x - b(u) \geq 0.
\end{aligned}
\end{align}
In short, when $F(x,u)$ is an affine map and $X(u)$ is a polyhedral set, the affine variational inequality
problem is equivalent to a linear complementarity problem over a larger
space of primal and dual variables. This can be more compactly stated as
the following monotone linear complementarity problem:
\begin{align} \label{generalized ULCP}
0 \leq z \perp B(u) z + d(u) \geq 0,
\end{align}
where $$ B(u) \triangleq \pmat{M(u) & -C(u)^T \\ C(u) & 0} \mbox{ and }  d(u)
	\triangleq \pmat{q(u) \\-b(u)},  $$ respectively.  It is relatively
	easy to see that $B(u)$ is a positive semidefinite matrix since
	$z^TB(u) z = x^TM(u) x \geq 0$ if $M(u)$ is a positive semidefinite
	matrix. This allows for making the following tractability claim when
	$\Uscr = \Uscr_2$. Naturally, we may also extend other statements
	drawn from the regime of uncertain linear complementarity problems but leave
	that for future work. 
\begin{proposition} Consider an uncertain variational inequality problem
denoted by $\{\mbox{VI}(X(u),F(\bullet;u))\}_{u \in \Uscr}$ where $\Uscr = \{u \mid \| u \|_2 \leq 1 \}$, $X(u)$
	and $F(x,u)$ are defined in \eqref{defFXu}, where
\begin{align*}
M(u) = S^T(u) S(u), S(u) = \sum_{l=1}^L
u_l S_l + S_0,q = q_0 + \sum_{l=1}^L u_l q_l ,\\ 
 C(u) = C_0 + \sum_{l=1}^L u_l C_l, b(u) = b_0 + \sum_{l=1}^L u_l b_l, u \in \Uscr.
\end{align*}
Then a robust solution of this problem is given by a solution to a tractable convex program.
\end{proposition}
\begin{proof}
Recalling that $z=(x,\lambda)$, the robust counterpart of
\eqref{generalized ULCP} is given by the following:
\begin{align}
\notag \min \qquad t\\
\label{sym-con} \st \quad  z^T ( B(u) z + d(u) )& \leq t, & \forall u \in \Uscr,\\
\label{asym-con} M(u) x  - C(u)^T
\lambda + {q(u)} & \geq 0, & \forall u \in \Uscr,\\
\label{asym-con-2} C(u) x  - b(u)  & \geq 0, & \forall u \in \Uscr,\\
\notag x, \lambda  & \geq 0.
\end{align}
We begin by considering constraint \eqref{sym-con} which can be recast
as as $z^T (B(u)+B(u)^T) z + 2z^Td(u) \leq 2t$. Consequently, this
constraint can be reformulated as a constraint similar to \eqref{cons1}.
By applying Lemma ~\ref{lemmacons1}, constraint \eqref{sym-con} can be
shown to be a tractable convex constraint.  Next, we proceed to show the tractability of constraint
\eqref{asym-con} which is equivalent to the following collection of $n$
optimization problems where $i = 1, \hdots, n$:
\begin{align} \notag
\min \quad & \left[\sum_{l=1}^L \left( [S_l^T S_0 + S_0^T S_l]_{i\bullet} x +
		{[q_l]}_i - {[C_l^T]_{i\bullet}} \lambda \right) u_l  + \sum_{l < m}
[S_l^T S_m + S_m^T S_l]_{i\bullet} x u_l u_m   +\sum_{l=1}^L [S_l^T S_l]_{i\bullet} x u_l^2 \right] \\
\st  \quad & \| u \|_2 \leq 1.
\end{align}
{Analogous to Theorem \ref{RC_unlcp_mon}, this problem can be
	rewritten as a linear matrix inequality and some linear inequlities. It follows that
		\eqref{asym-con} can be rewritten as a collection of $n$ linear
		matrix inequalities and a bunch of linear inequalitiess.}
Finally, constraint \eqref{asym-con-2} can be rewritten as the following
set of constraints:
$$
[C_0]_{i\bullet}x - {[b_0]_i} + \sum_{l=1}^L u_l \left([C_l]_{i\bullet} x -
{[b_l]_i}\right)\geq 0, \qquad \forall u \in \Uscr,  \qquad i = 1,
	\hdots, n.
$$
This set of semi-infinite constraints is equivalent to a finite set of
convex constraints in the form of second order cone constraints, which is discussed in Example 1.3.3. from \cite{bental09robust}.
\end{proof}

\subsection{Uncertain mathematical programs with complementarity
	constraints}
Over the last two decades, the mathematical program with equilibrium
constraints (MPECs) has found utility in modeling a range of problems,
including Stackelberg equilibrium problems, structural design problems,
bilevel programming problems, amongst others. A comprehensive
description of the models, theory, and the associated algorithms may be
found in the monograph by Luo et al.~\cite{luo96mathematical}. When the
lower-level problem is given by a complementarity problem, then the MPEC
reduces to a mathematical program with complementarity constraints
(MPCC). We consider the uncertain counterpart of MPCC defined as
follows:
\begin{align}\label{MPCC}
\begin{aligned}
\min & \quad   f(x,y)  \\
\st & \quad   h(x,y)   \geq 0 \\
	& \quad 0 \leq y  \perp F(x,y)  \geq 0. 
\end{aligned}
\end{align}
The MPCC is an ill-posed nonconvex program in that it lacks an interior. In fact,
	standard constraint qualifications (such as LICQ or MFCQ) fail to
	hold at any feasible point of such a problem.  We define an
	uncertain MPCC as a collection of MPCCs given by 
	$$ \{ \mbox{MPCC}(f,h,F) \}_{u \in \Uscr}, $$ 
	in which $f,h$ and $F$ are parametrized by
$u$ where $u \in \Uscr$:
\begin{align}\label{uMPCC}
\begin{aligned}
\min & \quad   f(x,y,u)  \\
\st & \quad   h(x,y,u)   \geq 0,   \\
	& \quad 0 \leq y  \perp F(x,y,u)  \geq 0. 
\end{aligned}
\end{align}
We may then define a robust counterpart of this problem as
follows:
\begin{align}\label{rMPCC1}
\begin{aligned}
\min & \quad   \max_{u \in \Uscr} f(x,y,u)  \\
\st & \quad   h(x,y,u)   \geq 0,  & \forall u \in \Uscr   \\
	& \quad 0 \leq y  \perp F(x,y,u)  \geq 0,   & \forall u \in \Uscr.
\end{aligned}
\end{align}
This problem is a nonconvex semi-infinite program. By utilizing the framework developed earlier, we may reformulate \eqref{rMPCC1} as a finite-dimensional MPCC.  Unfortunately, the semi-infinite complementarity constraint given by 
$$ 0 \leq y  \perp F(x,y,u)  \geq 0, \qquad \forall u \in
	\Uscr$$
need not admit a solution. Instead, we recast the uncertain
complementarity constraint as the following:
\begin{align}\label{rMPCC2}
\begin{aligned}
\min & \quad   t  \\
\st & \quad   f(x,y,u)  \leq t  \qquad \forall u \in \Uscr \\
	& \quad	h(x,y,u)    \geq 0  \qquad \forall u \in \Uscr \\
	& \quad  y   \mbox{ solves } \left\{ \begin{aligned}
				{\min_y} \max_{u \in U}  \quad & y^TF(x,y,u) \\
								F(x,y,u) & \geq 0, \qquad \forall u \in
								\Uscr\\
									& {y \geq 0} 
								\end{aligned}\right\}.
\end{aligned}
\end{align}
A natural question is whether a low-dimensional counterpart of \eqref{rMPCC2} is
available. Under convexity assumptions on $f(x,y,u)$ and concavity
assumptions on $h(x,y,u)$ in $x$ and $y$ for
every $u$, and some assumptions on the uncertainty set, tractable counterparts may be
constructed for the first two constraints in \eqref{rMPCC2}. By the findings of
the prior sections, under some conditions, a robust counterpart of an
uncertain LCP can be cast as a single convex program. The following
result presented without a proof provides a set of assumptions under which the lower-level problem
can be recast as a convex program:
\begin{proposition}
Suppose $F(x,y,u)$ is an affine map given by $F(x,y,u)=Ax + M(u)y +
q(u)$ and $M(u) = M_0 + \sum_{l=1}^L u_l M_l, q(u)=q, M_l \succeq 0,
	\forall l = 0, \hdots, L$ , then the third constraint of \eqref{rMPCC2} can be
	replaced by the optimality conditions of a convex program if $\Uscr
	= \Uscr_1, \Uscr_2$ or $\Uscr_{\infty}$.
\end{proposition}

\section{Numerical results}\label{sec:VI}
In Section~\ref{sec:VI-I}, we compare the quality of the residual of non-robust solutions with their robust analogues on an example presented in Section~\ref{sec:II}. The performance benefits of
the presented branching scheme on a non-monotone problem are examined in Section~\ref{sec:VI-II} and we
conclude with a case study on uncertain traffic equilibrium problems in
Section~\ref{sec:VI-III} where we compare robust solutions with the ERM
solutions investigated in the literature.
\subsection{Monotone uncertain LCPs}\label{sec:VI-I}
We consider the constructed uncertain LCP defined in Section~\ref{sec:II} 
for which the solution is known a priori.
Table~\ref{tab:ulcp} shows that the presented techniques allow for
obtaining the robust solution $x^{\texttt{rob}}$ and this corresponds
closely with the analytically available solution $x^{\texttt{anlyt}}$.
Furthermore, an arbitrarily chosen scenario-specific solution, given as
$x^{\texttt{n-rob}}$, leads to large deviation from the analytical optimal solution and 
significantly higher  residual.
\begin{table}[htbp]
{\scriptsize
\begin{center}
\begin{tabular}{|c|c|c|c|c|c|clcl}
\hline
n & $\|  x^{\texttt{rob}} - x^{\texttt{anlyt}} \|_2$ & residual of
$x^{\texttt{rob}}$ &
$\|  x^{\texttt{n-rob}} - x^{\texttt{anlyt}} \|_2$ & residual of
$x^{\texttt{n-rob}}$\\
\hline
10 & 3.9e-08 & 2.0e-07& 0.4e+03& 5.0e+07\\
\hline
20 & 4.7e-08 & 3.6e-07 & 0.7e+03 & 1.0e+09\\
\hline
40 & 1.8e-07 & 2.2e-06 & 1.6e+03 & 4.3e+10\\
\hline
80 & 5.1e-07 & 5.2e-06 & 5.5e+03 & 3.9+12\\
\hline
160 & 1.6e-05 & 5.3e-04 & 1.6e+04 &2.8e+14\\
\hline
\end{tabular}
\end{center}
}
\caption{Robust vs non-robust solutions}
\label{tab:ulcp}
\end{table}

\subsection{Non-monotone uncertain LCPs}\label{sec:VI-II}
We now consider a non-monotone LCP$(M(u),q(u))$ whose robust counterpart
is given by the following:
\begin{align}\label{ex3}
\begin{array}{rrl}
\min & t \\
\st & x^T (u_1S_1 - u_2S_2) x + (u_1q_1 + u_2q_2)^T x & \leq t, \quad \forall u \in \Uscr,\\
& (u_1S_1 - u_2S_2) x + u_1q_1 + u_2q_2 & \geq 0,\quad \forall u \in \Uscr,\\
& x & \geq 0.
\end{array}
\end{align}
where $e_n^T = (1,\hdots, n)$, $S_1 = e_ne_n^T \succeq 0$, $S_2 = 10^4 \times B^TB \succeq 0$, $\Uscr = \{(u_1,u_2) \mid 0
\leq u_1,u_2 \leq 1\}$, $B$ is a randomly generated matrix with elements drawn from ${\cal
		N}(0,1)$, $q_1 = -e_n$ and $q_2 = \frac{10}{n(n+1)} \times S_2e_n .$ It should
		be emphasized that our analysis allows for deriving the robust
		counterpart of this problem as a relatively low-dimensional nonconvex QCQP. In the absence
		of such an analysis, a direct approach would require solving an
		approximate nonconvex QCQP whose size is of the order of
		magnitude of the discretization.
Table~\ref{tab:nonconvex} provides a comparison of the performance of
three solvers on the RC  on a set of test problems for increasing matrix dimension $n$: (i) our branching scheme, (ii) the commercial
global optimization
solver \texttt{baron};  and (iii) the multi-start solver from
\texttt{Matlab}.
\begin{table}[htbp]
\centering {\scriptsize
\begin{tabular}{|c|c|c|c|c|c|c|c|c|c|c|}
\hline
Size & \multicolumn{4}{|c|}{Branching scheme} &
\multicolumn{2}{|c|}{\texttt{baron}} &
\multicolumn{2}{|c|}{\texttt{matlab}} & \\\hline 
$n$	& time(s) & $z^{\texttt{branch}}$ & Nodes & Gap & time(s) & $z^{\texttt{baron}}$ &
time(s) & $z^{\texttt{matlab}}$ &
	$ \frac{\|x^{\texttt{branch}}-x^{\texttt{baron}}\| } { 1+\|x^{\texttt{baron}}\| }$ &
	$\frac{\|x^{\texttt{branch}}-x^{\texttt{matlab}}\| }{ 1+\|x^{\texttt{baron}}\| }$ \\
\hline
6	&	1.32	&	0.0100	& 35 &	0.00	&	0.61	&	0.0098
&	3.52	&	2.54	&	0.00	&	0.33	\\
7	&	2.21	&	0.0035	 &	82	&	0.00	&	1.81
&	0.0164	 &	 4.54	&	2.53	&	0.00	&	0.21	\\
8	&	2.29	&	0.1648 &	87 &	0.00	&	0.87
&	0.1682	&	5.07	&	1.38	&	0.00	&	0.33\\
9	&	14.19	&	0.0072	&	406	&	0.00	&
0.89	&	0.0076	 &	 6.40 & 3.65	&	0.00 &	0.23	\\
10	&	9.80	&	0.0040	 &	254	&	0.00
&	1001.7	&	0.0125	 &	 6.93 & 0.74	& 0.00	&	0.22	\\
11	&	72.74 &	0.0036 &	893	&	0.00
&	2.94	&	0.0026	 &	6.09	&	14.31  &	0.00	&	0.22	\\
12 & 83.73 & 0.1998 & 1539 & 0.00 & 1.79 & 0.1990 & 8.22 & 10.69 & 0.00 & 0.13\\
\hline
\end{tabular}}
\caption{Global optimization of nonconvex QCQPs: CPU 3.40Ghz RAM 16.0 GB
}
\label{tab:nonconvex}
\end{table}
The results from Table \ref{tab:nonconvex} suggest the following. First,
	our branching scheme provides reasonably accurate solutions by
	comparing with the commercial solver \texttt{baron}, somtimes even better with respect to the optimal value $z$. Furthermore,
	the performance is significantly superior in terms of optimal value
	to the solutions provided by \texttt{Matlab}. {Third,
	\texttt{baron}'s performance in terms of time is superior to that
	provided by our \texttt{Matlab}-based branching solver is not
	altogether surpising,  given that it uses extensive
	pre-processing and has been developed on C/C++.} 
\subsection{Case study: Uncertain traffic equilibrium
	problems}\label{sec:VI-III}
\begin{figure}[htbp]
\begin{center}
\includegraphics[width=2in]{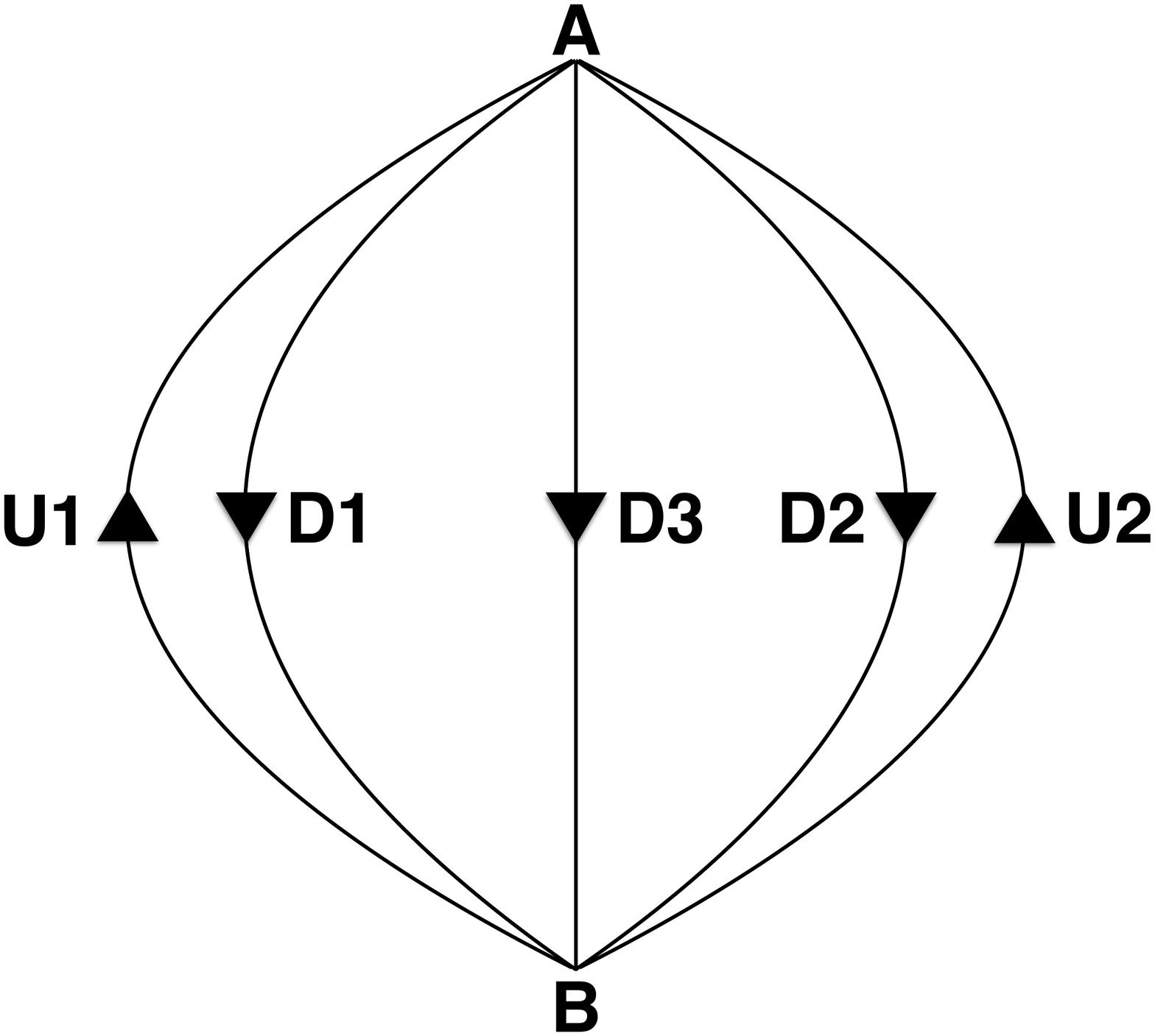}
\end{center}
\caption{2-node traffic network}
\end{figure}
\paragraph{2-node and 5 link network:} Consider the uncertain traffic equilibrium of the form described in
Section~\ref{sec:II}, sourced from~\cite{fang2007stochastic}. Suppose the associated network has two vertices $A,B$ and  five arcs
	$D_1,D_2,D_3,U_1,U_2$. Let $\xi$ denote the  flow over these five
	paths and $T(u) \xi + t$ represent the travel associated travel times,
	where $T(u)$ is an uncertain $5 \times 5$ matrix and $t \in \mathbb{R}^5$ is a
	constant vector. Suppose $B$ represents the path-OD pair incidence
	matrix and $d(u) \in
	\mathbb{R}^2$ represents the uncertain demand. Let $\tau$
	represent the minimum travel time for each direction.  Recall that
	the equilibrium point is given by a solution to the following:
$$
0 \leq x \perp M(u)x + q(u) \geq 0,\qquad \forall u \in \mathcal{U}
$$
where $x = (\xi,\tau)$, $M(u)$ and $q(u)$ are defined as
$$
M(u) =
\pmat{
T(u) & -B^T\\
B & 0},
q(u)=\pmat{t\\
-d(u)}, B=\pmat{1 & 1 & 1 & 0 & 0 \\
0&0&0&1&1}, t = \pmat{1000 \\ 950 \\ 3000 \\ 1000 \\ 1300}
,$$
and $T(u)$ is defined as 
\begin{align*}
T(u)  & = \pmat{40\alpha(u) & 0 & 0 & 20\beta(u) & 0 \\
0 & 60\beta(u) & 0 & 0 & 20\beta(u) \\
0 & 0 & 80\beta(u) & 0 & 0 \\
8\alpha(u) & 0 & 0 & 80\alpha(u) & 0 \\
0 & 4\beta(u) & 0 & 0 & 100\beta(u)},  
d(u)  = \pmat{260-100(\alpha(u)+\beta(u))\\
170-100(\alpha(u)+\beta(u))},
\end{align*}
$\alpha(u)=\frac{1}{2}u(u-1)$  and  $\beta(u)=u(2-u).$
Suppose $\mathcal{U} \triangleq \{ u_1, u_2, u_3 \}$ where $u_1, u_2,$
and $u_3$ denote a sunny, windy, and a rainy day respectively. In an
effort to compare the obtained solutions with that obtained from the ERM
model~\cite{fang2007stochastic}, we assume that these events occur with probability
$\frac{1}{2},\frac{1}{4},$ and $\frac{1}{4} $ (Note that our model does not
		require a probability distribution). Corresponding to this
problem, the ERM solution is denoted by $x^{\texttt{erm}}$ while the robust
solution is $x^{\texttt{rob}}$. Furthermore, non-robust scenario-specific
solutions are denoted by $x^1, x^2$ and $x^3.$ Table~\ref{tab:comp}
compares the optimality and feasibility of such points with respect to
the robust counterpart.
\begin{table}[htbp]
{\scriptsize
\begin{center}
\begin{tabular}{|c|c|c|c|c|c|clcl}
\hline
\multicolumn{4}{|c|}{solution} & infeasibility & complementarity\\
\hline
$x^1$ &  \multicolumn{3}{|c|}{$(0, 260, 0, 170, 0, 950, 1000)$} & 0 & 4.251E+06\\
\hline
$x^2$ &  \multicolumn{3}{|c|}{$(159.2, 0.83, 0, 70, 0, 1000, 1000)$} & 250 & 1.717E+06\\
\hline
$x^3$ &  \multicolumn{3}{|c|}{$(0, 160, 0, 3.75, 66.25, 950, 1300)$} & 500 & 2.228E+06 \\
\hline
$x^{\texttt{erm}}$ &  \multicolumn{3}{|c|}{$(84, 84, 21, 80, 20, 975, 1000)$} & 166 & 1.089E+06\\
\hline
$x^{\texttt{rob}}$ &  \multicolumn{3}{|c|}{$(117.7, 89.5, 52.8, 90.5, 79.5, 950, 1000)$} & 0 & 1.840E+06\\
\hline
\end{tabular}
\end{center}
}
\caption{Comparison across solutions}
\label{tab:comp}
\end{table}
In this table, the infeasibility function is defined as
${\displaystyle \max_{u \in \mathcal{U}}}(e^T\max(- M(u)x - q(u), 0))
$
while the complementarity residual is defined as:
$
{\displaystyle \max_{u \in \mathcal{U}}} x^T ( M(u) x + q(u) ).
$
It is seen that the robust solution and $x^1$ are feasible for every
$u$. Notably, the ``sunny day'' design is feasible but leads to a large  
complementarity residual. Note that if both the
feasibility and the complementarity metric is zero, this implies that
the solution is an equilibrium for every $u \in \Uscr$. {The robust
solution minimizes the worst complementarity residual among all possible scenarios and from that standpoint, it is seen to be superior to $x^1$, the solution that minimizes the residual for the first scenario.} Furthermore, $x^{\texttt{erm}}$ might have a superior complementarity residual but such a solution may be rendered infeasible for certain realizations.  
Table~\ref{tab:comp1} compares the value of parametrized gap function
$G(x,u)$, defined as 
$$G(x,u) \triangleq {\displaystyle\sup_{y\geq 0}} (x-y)^T (M(u)x +
		q(u)).$$ The
lowest value of $G(x,u)$ is achieved by $x^{\texttt{erm}}$ when $u =
u_2$. However, $G(x^{\texttt{erm}},u_1) = G(x^{\texttt{erm}},u_3) =
	+\infty$, a consequence of infeasibility.   However,
	$G(x^{\texttt{rob}},u) < \infty$ for every $u \in \Uscr$. In  
Table~\ref{tab:comp2}, we consider how the robust solution satisfies
demand requirements (ensuring feasibility) while the ERM solution may
not satisfy demand for all realizations (leading to infeasibility).
\begin{table}
{\scriptsize 
\begin{minipage}[b]{0.4\linewidth}\centering
\begin{tabular}{|c|c|c|c|c|c|clcl}
\hline
& $u=u_1$ & $u=u_2$ & $u=u_3$ \\
\hline
$x^{\texttt{rob}}$ &	1.4e+5 & 1.8e+6 & 1.8e+6 \\
\hline
$x^{\texttt{erm}}$ & Inf	 & 6.5e+5 & Inf \\
\hline
\end{tabular}
\label{tab:comp1}
\caption{Evaluation of  $G(x^*,u)$}
\end{minipage}
\begin{minipage}[b]{0.6\linewidth}\centering
\begin{tabular}{|c|c|c|c|c|c|c|}
\hline
OD pair	& possible demand & $x_{RO}$	& $x_{ERM}$ & $x_1$ &	$x_2$ &	$x_3$ \\
\hline
AB & 260,160 & 260 & 189 & 260 &	160 & 160\\
\hline
BA & 170,70 & 170 &	100	& 170 & 70 & 70\\
\hline
\end{tabular}
\label{tab:comp2}
\caption{Flow of each OD pair}
\end{minipage}}
\end{table}

\paragraph{5-node and 7-link network:} We now consider a larger traffic network considered in
Section~\ref{sec:II} with 7 links and 6 paths. 
\begin{figure}
\vspace{-0.1in}
\begin{center}
\includegraphics[width=2in]{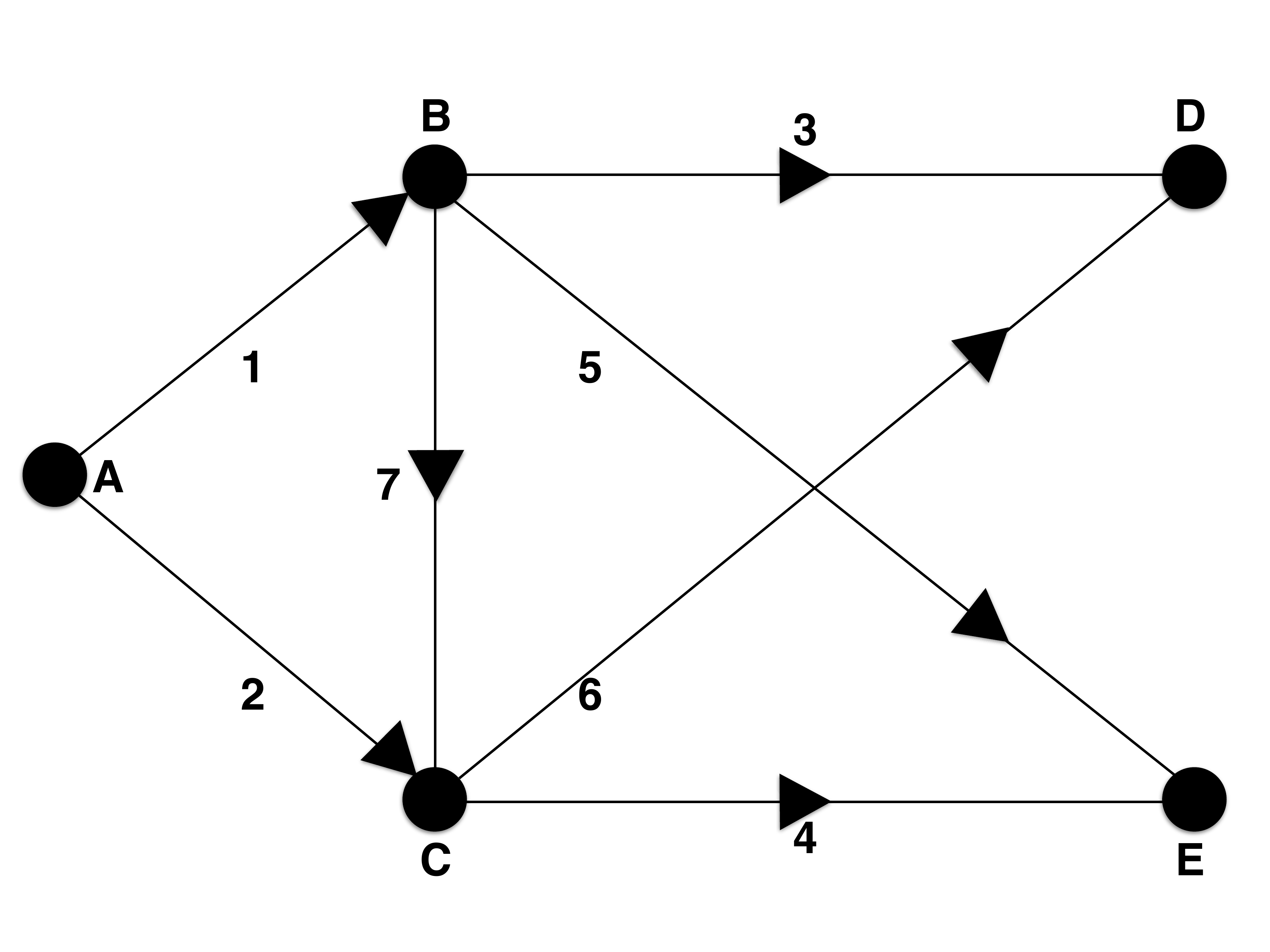}
\end{center}
\caption{Traffic Network}
\label{traff-net-1}
\end{figure}
{Figure ~\ref{traff-net-1} represents a 7-link network with 6-paths
sourced from~\cite{ChenWZ12} and $A \rightarrow D$ and $A \rightarrow E$
represent two origin-destination (OD) pairs. The OD pair $A \rightarrow D$ is
connected by paths $p_1=\{1,3\},p_2=\{1,7,6\},p_3=\{2,6\}$ while the OD
pair $A \rightarrow E$ is connected by paths $p_4=\{1,5\}, p_5 = \{1,7,4\}, p_6
=\{2,4\}$. The demand along every OD pair is denoted by  $d(u)\in
\mathbb{R}^2$ where $u \in \mathcal{U}$ while the link capacity is
captured by the vector $c(u)\in \mathbb{R}^7,u \in \mathcal{U}$. Let
vector $x \in \mathcal{R}^6$ denote the assignment of flows to all path
from $p_1$ to $p_6$ and $f \in \mathcal{R}^7$ denote the assignment of
flows to all links $1,\hdots,7$. Then the relationship between $x$ and
$f$ is presented by:
$
f = \Delta x,
$
where $\Delta = (\delta_{i,j})$ is the link-path incidence matrix. The
entry $\delta_{i,j}$ is set at $1$ if and only if link $i$ lies in path
$j$. Let $B=(b_{i,j})$ denote the OD-path incidence matrix and
$b_{i,j}=1$ if and only if path $j$ connects the $i$th OD pair. In this
	case, the two matrices are given as follows:
\begin{equation}
\Delta =
\left(
\begin{array}{cccccc}
1 & 1 & 0 & 1 & 1 & 0\\
0 & 0 & 1 & 0 & 0 & 1\\
1 & 0 & 0 & 0 & 0 & 0\\
0 & 0 & 0 & 0 & 1 & 1\\
0 & 0 & 0 & 1 & 0 & 0\\
0 & 1 & 1 & 0 & 0 & 0\\
0 & 1 & 0 & 0 & 1 & 0\\
\end{array}
\right) \mbox{ and }
B =
\left(
\begin{array}{cccccc}
1 & 1 & 1 & 0 & 0 & 0 \\
0 & 0 & 0 & 1 & 1 & 1 \\
\end{array}
\right).
\end{equation}
Following a generalized bureau of public roads (GBPR) function, the multivalue link cost function $C(f,u)$ is defined as:
\begin{equation}
C_i(f,u)=c_i^0 \left(1.0 + 0.15 \left( \frac{f_i}{c_i(u)} \right)^{n_i} \right), i = 1, \hdots, 7
\end{equation}
where $c_i^0$ and $n_i$ are known parameters. Let $n_i = 1$ for all $i$, then the travel cost function is given as:
\begin{align}
F(u,x) & =\eta \Delta^T C(\Delta x,u)  = 0.15 \eta \Delta^T \mbox{diag} \left( \frac{c_i^0}{c_i(u)} \right) \Delta x + \eta \Delta^T c^0 \triangleq M(u)x + q.
\end{align}
Let $w \in \mathbb{R}^2$ denote the minimum travel cost of each OD pair.
	Last, by Wardrop's user equilibrium, the uncertain CP formulation is
	given by the following:
\begin{equation}\label{trafficequLCP}
0 \leq \pmat{x \\ w} \bot \pmat{M(u) &  -B^T \\ B & 0} \pmat{x \\ w} +
	\pmat{q \\ -d(u)} \geq 0, \qquad \forall u \in \Uscr.
\end{equation}}
Tables \ref{tab:comp3} and 
\ref{tab:comp4} show the comparison between different solutions of the
LCP given by \eqref{trafficequLCP}. $x^{\texttt{rob}}$ denotes the
robust solution  of \eqref{trafficequLCP} in that it minimizes the worst
case of the gap function $G(z,u)$, defined as 
$$G(z,u) = \sup_{y\geq
	0}(z-y)^T(A(u)z+b(u)), z = \pmat{x \\ w}, A(u) = \pmat{M(u) & -B^T
		\\ B & 0}, b(u) = \pmat{q \\ -d(u)}.$$ We consider a
case when $\Uscr = \{ u \mid -1 \leq u \leq 1\},
\frac{1}{c_i(u)} = (\widehat{c}_0)_i + u (\widehat{c}_1)_i$, where
$\widehat c_0$ and $\widehat c_1$ are defined as follows:
\begin{align*}
\widehat{c}_0 & = -\widehat{c}_1=
(1/40,1/40,1/20,1/20,1/20,1/20,1/20), c^0 = (3,5,6,4,6,4,1),\\
d(u) & = d_0 + ud_1, d_0=(200;220), d_1=(50;40).
\end{align*}
{The ERM solution $x_{ERM}$ is
constructed as follows. Let $x_{ERM}=(y;w)$ where $y$ is obtained by
	$$y = (I - B^\dagger B)x^* + B^\dagger \mathbb{E}[d(u)], \mbox{
		where }  B^\dagger = B^T(B B^T)^{-1}.$$ Note that $x^*$ is a
		minimizer of $\phi(x)$ over the set $D$, where
\begin{align*}
\phi(x)  & = \mathbb{E}[f(x,u)], f(x,u) = z(x,u)^T F(z(x,u),u) + Q(z(x,u),u),
z(x,u) = (I - B^\dagger B)x + B^\dagger d(u),\\
F(z,u) & = M(u)z + q,
Q(z,u)=\max \left\{ -y^T F(z,u) \mid By = d(u), y \geq 0 \right\} = \min
	\left\{ y^T d(u) \mid B^T y + F(z,u) \geq 0\right\},\\
D & = \left\{ x \mid B^\dagger B x \leq c, c_i = \min_{u \in \Uscr}
(B^\dagger d(u))_i \right\},
\end{align*}
as per the recent work by  Chen, Wets and Zhang~\cite{ChenWZ12}. 
Note that an estimator the minimizer of $\mathbb{E}[f(x,u)]$ is obtained
via sample-average approximation schemes while $w$ is acquired by taking the minimum of the
average costs of paths in each OD-pair. Let the average costs of paths be captured by a vector $ v =
\mathbb{E} [M(u)y+q ]$, then $w_1 =
\min\{v_1,v_2,v_3\},w_2 =
\min\{v_4,v_5,v_6\}$. $x_1,\hdots,x_5$ are
		the solutions of the program $\min_{A(u)x+b(u) \geq 0,x\geq 0}
		\{x^T(A(u)x+b(u))\}$, when $u = -1,-0.5,0,0.5,1$, respectively.} 
	\begin{table}[htbp]
{\scriptsize
\begin{center}
\begin{tabular}{|c|c|c|c|c|c|c|c|c|c|}
\hline
OD pair	& range of possible demand & $x^{\texttt{rob}}$	&
$x^{\texttt{erm}}$ &$x_1$ &	$x_2$ &	$x_3$ & $x_4$ & $x_5$\\
\hline
AD & 150-250 & 250 & 200 &	150 & 175 & 225 & 250 & 200\\
\hline
AE & 180-260 & 260 & 220 & 180 & 200 & 240 & 260 & 220\\
\hline
\end{tabular}
\end{center}
}
\caption{Flow across two OD pairs}
\label{tab:comp3}
\end{table}
Table \ref{tab:comp3} shows the traffic flow between two OD pairs.
Again, the robust solution satisfies  the largest possible
demand while the ERM solution does not satisfy demand for all possible
realizations. When we compare the residual function for a particular
$x$ and $u$ as seen in Table \ref{tab:comp4}, {while the robust solution
$x^{\texttt{rob}}$ does not provide the best  function value for every scenario, it minimizes the
worst case. In fact, for the non-robust solutions, except $x^4$, every
solution displays an infinite residual function for some $u$. Notably,
		 the ERM solution also have infinite residuals for some realizations of $u$.} 
\begin{table}[htbp]
{\scriptsize
\begin{center}
\begin{tabular}{|c|c|c|c|c|c|c|c|c|c|}
\hline
$u$ & $x^{\texttt{rob}}$	& $x^{\texttt{erm}}$ & $x^1$ &	$x^2$ &	$x^3$ & $x^4$ & $x^5$\\
\hline
-1 & 10343 & 4340  & 	6488 &	7922	& 18322 &	19000	& 4329\\
\hline
-0.5 & 7863 & 2176  & Inf	 & 6772 & 14671 & 14250 &	2165\\
\hline
0.0	& 5382 & Inf & Inf	 & Inf &	11021	& 9500 &	0.000449\\
\hline
0.5 &	2901 & Inf & Inf	& Inf &	7370 &	4750 &	Inf \\
\hline
1.0 & 421 & Inf & Inf	& Inf	& Inf &	1.170e-05	& Inf\\
\hline
\end{tabular}
\end{center}
}
\caption{Residual function value at different sample points}
\label{tab:comp4}
\end{table}


\section{Concluding remarks} \label{sec:VII}
\begin{table}[H]
{\scriptsize
\begin{center}
\begin{tabular}{|c|c|c|c|c|c|clcl}
\hline
$\Uscr$ & $M,q(u)$ & $M(u),q$ & $M(u),q(u)$ & $A^T(u)A(u),q(u)$\\
\hline
$\Uscr_\infty$ & convex QP & convex QP & nonconvex QCQP & $\diagup$\\
\hline
$\Uscr_1$ & convex QP & convex QCQP & nonconvex QCQP & $\diagup$\\
\hline
$\Uscr_2$ & convex QCQP & convex program & nonconvex program & SDP\\
\hline
$\Uscr_c$ &  QP with conic constraints & $\diagup$ & $\diagup$ & $\diagup$\\
\hline
$\Uscr_1^r$ & $\diagup$ & convex QCQP & $\diagup$ & $\diagup$ \\
\hline
$\Uscr_\infty^r$ & $\diagup$ & convex QP & $\diagup$ & $\diagup$\\
\hline
\end{tabular}
\end{center}
}
\caption{Characterization of robust counterparts under varying
	assumptions. }
\label{tab: conclusion}
\end{table}

In this paper, we consider the resolution of finite-dimensional monotone
complementarity problems corrupted by uncertainty. A distinct thread in
the literature has considered the minimization of the expected residual.
This avenue relies on the availability of a probability distribution and
the solution of a stochastic, and possibly nonconvex, program. Instead,
we consider an avenue that relies on the availability of an uncertainty
set. By leveraging findings from robust convex programming, we show that
uncertain monotone linear complementarity problems can be tractably
resolved as a single convex program. {In fact, when the uncertain linear
complementarity problem is not necessarily monotone, under some
conditions on the uncertainty set, the tractable robust counterpart of
this problem can be shown to be convex, a consequence of leveraging the
hidden convexity in the problem. More generally, the robust
counterpart is a nonconvex  quadratically constrained quadratic program.
We adapt and present a recently presented branching scheme to accommodate such
problems. Table \ref{tab: conclusion} provides a compact representation of the
tractability statements and the nature of the uncertainty sets that correspond
to these statements. The columns of this table correspond to different
assumptions of uncertainty on $M$ and $q$. Note that $M(u) = M_0 +
\sum_{l=1}^Lu_lM_l, M_l \succeq 0 \mbox{ or } \preceq 0$ for $l = 1,\hdots,L$,
$q(u) = q_0 + \sum_{l=1}^Lu_lq_l$, and $A(u) = A_0 + \sum_{l=1}^Lu_lA_l.$ We further
observe that such statements can be utilized to show {the following}:
\bi
\item[(1)]
The tractable robust
counterparts of an uncertain affine variational
inequality problem(uncertain AVIs) over uncertain polyhedral sets are SDPs under some assumptions on the uncertainty set.
\item[(2)]
Robust counterparts of mathematical programs with uncertain linear complementarity constraints (uncertain MPCC) can be reformulated as deterministic low-dimensional mathematical programs with complementarity constraints.
\ei}
{Future research will consider a study of nonlinear generalizations of
	$F(x,u)$ as well as extensions to variational and hierarchical
		regime.} 
Finally, our preliminary numerical investigations {reinforce the
belief} that non-robust solution {may} produce large worst-case
residual compared with robust solutions. Furthermore, we present a 
branching-based procedure for {obtaining global solutions to robust
	counterparts of  non-monotone uncertain LCPs and note its
	effectiveness compared to commercial global solvers.}  Finally,
	robust solutions are qualitatively different
from their ERM counterparts in the context of traffic
equilibrium problems.

{Finally,  we comment on the nature of the uncertainty sets. We employ an ``uncertain-but-bounded'' model of uncertainty (a terminology that has its roots in~\cite{bental09robust}) in which the values of the uncertain parameter are represented through an affine parametrization of $\zeta$ which varies in the perturbation set $\mathcal{Z}.$ Naturally, this is by no means the only way to represent uncertainty. For instance, one alternate approach is to introduce chance constraints with {\em ambiguity}; in this model, we impose chance or probabilistic constraints under the caveat that the distribution is known partially in that it belongs to a family of distributions. We leave such questions for future work.  }

 {\bf Acknowledgements.} We would like to thank the associate editor and the two referees for their constructive suggestions. In particular, their comments have led to the addition to the NP-hardness result (Lemma 4.1) and the tractability result (Theorem 3.10). We would also like to thank Dr. Vineet
Goyal for alerting us to Reference [4]. 
\bibliographystyle{siam}

\end{document}